\documentclass[a4paper, 12pt,reqno]{article}
\usepackage[utf8]{inputenc}
\usepackage{babel}
\usepackage{natbib}
\usepackage{authblk}
\usepackage{hyperref}
\hypersetup{
    colorlinks,
    linkcolor={blue!30!black},
    citecolor={blue!50!black},
    urlcolor={blue!80!black}
}
\usepackage{microtype}

\usepackage{geometry}
\usepackage{graphicx}%
\usepackage{multirow}%
\usepackage{amsmath,amssymb,amsfonts}%
\usepackage{amsthm}%
\usepackage{mathrsfs}%
\usepackage{xcolor}%
\usepackage{textcomp}%
\usepackage{manyfoot}%
\usepackage{booktabs}%
\usepackage{algorithm}%
\usepackage{algorithmicx}%
\usepackage{algpseudocode}%
\usepackage{listings}%
\usepackage{caption}
\usepackage{subfigure}

\newtheorem{theorem}{Theorem}[section]
\newtheorem{proposition}{Proposition}[section]

\newtheorem{example}{Example}[section]%
\newtheorem{remark}{Remark}[section]%

\newcommand{\mb}{\mathbf}

\usepackage{xspace}

\newcommand\ie{\emph{i.e.}\xspace}
\newcommand\iid{\ensuremath{\mathit{i.i.d.}}\xspace }
\newcommand\iidsim{\ensuremath{\overset{\mathit{i.i.d.}}{\sim}\xspace }}
\newcommand\eg{\emph{e.g. }\xspace}

\newcommand{\wrt}{\emph{w.r.t.}\xspace}

\newcommand{\rset}{\mathbb{R}}

\newcommand{\ind}{\mathbf{1}}
\newcommand{\un}{\ind}
\newcommand{\1}{\ind}
\newcommand{\point}{\,\cdot\,}
\newcommand{\eqd}{\overset{d}{=}}
\usepackage{xifthen}
\newcommand{\PP}[1][]{\ifthenelse{\equal{#1}{}}{\ensuremath{\mathbb{P}}}{\ensuremath{\mathbb{P}\left( #1 \right)}}} 
\newcommand{\EE}[1][]{\ifthenelse{\equal{#1}{}}{\ensuremath{\mathbb E}}{\ensuremath{{\mathbb E}\left[ #1 \right]}}}
\newcommand{\given}[1][]{\ifthenelse{\equal{#1}{}}{\ensuremath{\;\vert\;}}{\ensuremath{\;\left|\; #1 \right. }}}
\newcommand{\law}{\mathcal{L}}

\newcommand{\ud}{\,\mathrm{d}}
\newcommand{\wto}{\overset{\mathrm{w}}{\to}}
\DeclareMathOperator{\minimize}{\mathrm{minimize}}
\DeclareMathOperator{\argmin}{\mathrm{argmin}}

\newcommand{\sphere}{\mathbb{S}}
\newcommand{\ball}{\mathbb{B}}
\newcommand{\spherepos}{{\mathbb{S}_{+}}}

\newcommand{\cl}{\mathrm{cl}}

\newcommand{\sspace}{\mathcal{X}}
\newcommand{\class}{\mathcal{A}}
\newcommand{\vapdim}{\mathcal{V}}
\newcommand{\shatter}{\mathcal{S}}

\newcommand{\borel}{\mathcal{B}}
\newcommand{\dd}{\{1,\ldots,d\}}

\newcommand{\EVT}{\textsc{evt}\xspace}

\newcommand{\ERM}{\textsc{erm}\xspace}

\DeclareMathOperator{\polar}{Polar}

\newcommand{\eps}{\varepsilon}
\newcommand{\cone}{\mathcal{C}}

\newcommand{\Oh}{\operatorname{O}}

\newcommand{\vertiii}[1]{{\left\vert\kern-0.25ex\left\vert\kern-0.25ex\left\vert #1
    \right\vert\kern-0.25ex\right\vert\kern-0.25ex\right\vert}}

\newcommand{\mv}{\textsc{mv}\xspace}

\newcommand{\hatphi}{\widehat{\Phi}}

\newcommand{\HTalgo}{{\fontfamily{qcr}\small{LHTR}}\xspace}
\newcommand{\lhtr}{\HTalgo}

\newcommand{\risk}{\mathcal{R}}
\newcommand{\riskhat}{\widehat{\mathcal{R}}}
\newcommand{\rCV}{\riskhat_{\mathrm{CV}}}
\newcommand{\rcvEx}{\rCV}
\newcommand{\hatg}{\widehat{g}}

\newcommand{\bayesext}{f_{P_\infty}}

\newtheorem{lemma}{Lemma}[section]
\newtheorem{assumption}{Assumption}[section]

\graphicspath{{figures}}
\begin{document}

\title
{Weak Signals and Heavy Tails: Learning Theory meets Extreme Value Analysis}



\author[a]{Stephan Clémençon}
\author[b]{Anne Sabourin}

\affil[a]{ LTCI, T\'el\'ecom Paris, Institut Polytechnique de Paris, Palaiseau, France}
\affil[b]{Universit\'e Paris Cit\'e, Universit\'e Paris Saclay, ENS Paris Saclay, CNRS, SSA, INSERM, Centre Borelli, F-75006, Paris, France}




\maketitle
\abstract{The masses of data now available have opened up the prospect of discovering weak signals using machine-learning algorithms, with a view to predictive or interpretation tasks. As this survey of recent results attempts to show, bringing multivariate extreme value theory and statistical learning theory together in a common, nonparametric and nonasymptotic framework makes it possible to design and analyze new methods for exploiting the scarce information located in distribution tails in these purposes. This article reviews recently proved theoretical tools for establishing guarantees for supervised or unsupervised algorithms learning from a fraction of extreme data. These are mainly exponential maximal deviation inequalities tailored to low-probability regions and concentration results for stochastic processes empirically describing the behavior of multivariate extreme observations, their dependence structure in particular. Under appropriate assumptions of regular variation, several illustrative applications in multivariate settings are then examined: classification, regression, anomaly detection, model selection via cross-validation. For these, generalization results are established inspired by the classical bounds in statistical learning theory. In the same spirit, it is also shown how to adapt the popular high-dimensional Lasso technique in the context of extreme values for the covariates with generalization guarantees.
}

\section{Introduction}\label{sec1}



Traditionally, Extreme Value Theory (\EVT) has been grounded in an asymptotic framework, reflecting its probabilistic foundations. In the Peaks-Over-Threshold approach, thresholds tend to infinity, while block-maxima methods rely on increasing block sizes. From a modeling perspective, parametric assumptions have dominated multivariate \EVT, justified by the need to mitigate data scarcity through strong structural constraints. A brief introduction to \textsc{EVT}—with a focus on the multivariate setting—and key references are provided in Section~\ref{sec:RV-back}.

In stark contrast, statistical learning theory, which underpins predictive machine learning and artificial intelligence, operates within a non-asymptotic, non-parametric paradigm. This literature delivers universally valid error bounds for "learnable" algorithms, applicable to finite sample sizes and independent of the data distribution, provided complexity assumptions hold for the class of predictive functions. These bounds are typically derived using concentration of measure inequalities or combinatorial results from Vapnik--Chervonenkis (VC) theory. Further details are deferred to Section~\ref{sec:LT-back}.

\medskip

\noindent
{\bf Purpose \& contributions.}
The growing availability of large datasets enables the use of data-intensive machine learning algorithms in the context of extreme value analysis, bridging the gap between these two domains.
The purpose of this paper is to provide  a unified framework that integrates and extends a series of prior contributions by the authors and colleagues, 
demonstrating  that the  opposition between \EVT and statistical  Machine Learning is not insurmountable, in either theory or practice. 
The primary focus is on statistical learning theoretical results, 
with discussions of applications.  Our contribution is double. On the one hand we provide a unified  overview of several  earlier works from the authors  regarding nonasymptotic statistical learning guarantees for \EVT, with some machine learning applications. Through this review, we seek to demonstrate the feasibility and effectiveness of integrating \EVT with modern statistical learning techniques. A significant portion of the methods presented here has been implemented in the Python package \texttt{MLExtreme}.\footnote{\url{https://github.com/hi-paris/MLExtreme}, available on PyPI.}
On the other hand, we present novel results  concerning complexity regularization (penalized empirical loss minimization). These new results extend the stylized framework developed in earlier works to potentially high-dimensional settings and  further illustrates the wide applicability of these works.

\medskip

\noindent
{\bf Scope.} 
A central paradigm in the developments presented here is the formulation of methods based on Empirical Risk Minimization (ERM), a foundational concept in learning theory and artificial intelligence. The applications motivating this theory span diverse contexts in multivariate settings, in unsupervised and supervised settings,  including anomaly detection, generative models for natural language processing, and more traditional \EVT applications, such as delineating risk regions based on streamflow data. In supervised settings, the main focus is on prediction based on extreme values of the covariates. 

The algorithms and methods discussed here are supported by finite-sample error bounds, which typically scale as \(1/\sqrt{k}\), where \(k\) is the number of observations retained as `extreme' in the training step. These bounds are derived by decomposing the error into a bias term, arising from the finite-distance nature of the data, and a variance term, capturing deviations from the mean conditional on an excess. While the bias term is often excluded from the analysis, regular variation assumptions ensure its vanishing above sufficiently large thresholds.

This framework provides a precise characterization of \emph{weak signals} detectable by machine learning in large datasets. Unlike classical algorithms, which primarily capture statistical regularity near the center of the distribution, the approaches discussed here exploit the rarer, tail-located information. This objective introduces new trade-offs, governed by the assumption of multivariate regular variation. Statistical learning in tail regions necessitates accounting for additional sources of bias, with significant implications for both theory and practice to ensure the validity and effectiveness of the frequentist ERM principle in this context.

This review only briefly addresses the growing field of dimension reduction for multivariate extremes, due to space constraints and its divergence from the core focus: leveraging concentration and VC-type inequalities to derive guarantees for ERM algorithms.

\medskip

\noindent {\bf Related work.} As explained above, this article focuses on recent developments in statistical learning theory that enable the analysis of predictive performance in tail regions of high dimension. It should be noted, however, that in recent years, a great deal of work has been done to apply empirical machine learning approaches, deep learning techniques in particular, to the analysis of multivariate extreme values. Without claiming to be exhaustive, one can cite the following articles. The use of neural networks to model and simulate extreme events, addressing in some cases parameter adjustment and architecture selection, is the subject of the articles \cite{bhatia2021exgan}, \cite{hasan2022modeling}, \cite{10650414}, \cite{de2025kolmogorov}, \cite{richards2024neural}, \cite{de2025generative}, \cite{JMLR:v23:21-0663}, \cite{dahal2024junction}, \cite{lafon2023vae}, \cite{mackay2024deep} for instance, see also \cite{wessel2025comparison} for a recent benchmark. Regarding the topic of dimensionality reduction in tail regions, one can mention the contributions in \cite{goix2016sparse}, \cite{goix2017sparse}, \cite{meyer2019sparse}, \cite{gardes2018tail}, \cite{engelke2020sparse}, \cite{aghbalou2021tail}, \cite{butsch2024information}, \cite{gardes2025dimension}, \cite{girard2025extremeplsmissingdataweak}, as well as \cite{chiapino2016feature,chiapino2019identifying}, \cite{chiapino2019vizu}, which deal with various issues such as sparse support estimation, extreme quantile or conditional tail-index estimation, tail inverse regression or unsupervised feature structure detection. One may refer to \cite{janssen2019k}, \cite{meyer2024multivariate}, \cite{JMLR:v25:21-1367} concerning the clustering of multivariate extremes, to \cite{drees2019principal}, \cite{jiang2020principal}, \cite{clemenccon2023regular}, \cite{butsch2025estimation}, \cite{cooley2019decompositions},\cite{jiang2020principal} for latent variable analysis and related model selection issues, to \cite{engelke2020graphical}, \cite{engelke2021learning}, \cite{doi:10.1137/24M1678635}, \cite{boulin2025structuredlinearfactormodels}, \cite{constantinou+d:2017}, \cite{mhalla2020causal}, \cite{gnecco2021causal} for graphical models and causal inference. More generally, the predictive power of variants of popular machine-learning methods adapted to be effective in tail regions has been recently investigated in numerous papers, too many to cite here, see \eg  \cite{velthoen2021gradient}, \cite{gnecco2024extremal}, \cite{verma2025optimal}, \cite{deCarv22}, \cite{buritica2024progressionextrapolationprincipleregression}.

\medskip
 
\noindent
{\bf Outline.}
After presenting in Section~\ref{sec:background} some background with key concepts and references both in \EVT and in statistical learning theory, 
Sections~\ref{sec:unsupervised} and~\ref{sec:supervised} review  recent advancements that contribute to reconciling \EVT with statistical machine learning, 
in unsupervised and supervised frameworks.  
Some new material is presented in Section \ref{sec:regression} that highlights the close relationships between the traditional regular variation framework and the `Learning on extreme covariates' setting, thereby opening the way to numerous applications of the latter. Section~\ref{sec:exlasso_section} presents novel results. Specifically, we  develop a natural extension of the least squares methods discussed in Section~\ref{sec:regression} to a penalized problem in a high-dimensional context, specifically a variant of the Lasso. We demonstrate that some existing guarantees on the standard Lasso, which have become standard in the realm of linear models with sub-Gaussian or bounded noise, carry over to this extension under appropriate assumptions regarding the tail dependence structure between the covariate and the target. Finally, Section~\ref{sec:perspectives} gathers some concluding remarks.





\section{Background and Preliminaries}\label{sec:background}

  This section introduces notation and 
  minimal background in multivariate \EVT and  statistical learning theory.

\subsection{Notations} 
Here and throughout, the indicator function of any event $\mathcal{E}$, valued in $\{0,1\}$, 
is denoted by $\un_{\mathcal{E}}$, and when the considered event corresponds to some condition $\mathcal{C}(x)$ involving some variable $x$ being satisfied, we write $\un\{\mathcal{C}(x)\}$. The $d$-dimensional Euclidean space $\rset^d$ is endowed with
its Borel $\sigma$-field $\mathcal{B}(\rset^d)$ and is equipped with a
norm $\|\cdot\|$. With respect to it, by $\ball$ is meant the unit
open ball in $\rset^d$, by $\sphere$ the unit sphere of $\rset^d$ and
by $\spherepos$ its intersection with the positive orthant
$\rset_+^d$. Vectors in  $\rset^d$ are denoted in bold. 
 The left-continous inverse of any nondecreasing
c\`adl\`ag function $H:\rset\to \rset$ is denoted by
$H^{-1}$. For $Z$ a random object we sometimes denote the distribution
of $Z$ by $\law(Z)$, and by $\law(Z\given \mathcal{E})$ the conditional distribution of $Z$ given the event $\mathcal{E}$.
The notation $(Z_i)_{i\le n} \iidsim Z$  means that the $Z_i$'s are independent and identically distributed (iid)  copies of $Z$. 
Algebraic operations between vectors on $\rset^d$ are understood componentwise, unless otherwise stated,  and  if $\mb x, \mb z$ are in $\rset^d$, the condition $\mb x \le \mb z$ means $X_j\le x_j$ for
 all $j\le d$; while $\mb X\not\le \mb x$ is the negation of the previous
 condition.
If $A\subset \rset^d$ and $t\in\rset$, then $tA$ is the set $\{t\, {\bf x}, {\bf x}\in A \}$, $\cl(A)$ is the closure of $A$, and $\partial A$ is the boundary of $A$. Convergence in distribution of random elements $(Z_n, \, n\ge 1)$, to a nondegenerate limit $Z_\infty$  (\ie weak convergence) is denoted by $Z_n\wto Z_\infty$. 


\subsection{Multivariate Extremes and Regular Variation}\label{sec:RV-back}
Most of the material presented in this paper focuses on learning
problems in multivariate (and  possibly high dimensional) spaces, typically $\rset^d$ when
$d > 1$. We consider a random vector (r.v.)
$\mb X = (X_1, \ldots, X_d)$ valued in $\sspace\subset\rset^d$ with probability distribution $P$,   and $n\geq 1$ \iid\ replications of
it: $\mb X_i = ( X_{i,1}, \ldots,X_{i,d}), 1\le i\le n$. 
A traditional assumption in \EVT, is that after
a suitable marginal standardization to unit Pareto margins, the
 distribution of the standardized vector $\mb V$
(see~\eqref{eq:standardize-pareto} below) conditionally on  $\|\mb V\|>t$, 
converges to a certain limit as $t\to \infty$. Precisely, denoting by
$F$ the cumulative distribution of $\mb X$ and letting
$F_j(u) = \PP(X_j \le u)$ for $u\in \rset$, define
\begin{equation}
 \label{eq:standardize-pareto}
 \begin{aligned}
   v(\mb x)  = \Big(\frac{1}{1 - F_1(x_1)}, \ldots, \frac{1}{1 - F_d(x_d)} \Big) \text{ for } \mb x=(x_1, \ldots, x_d)\;~\text{ and } 
    \mb V = v(\mb X) .  
    \end{aligned}
 \end{equation}
 A key assumption is the existence of a limit Borel measure $\mu$ on
 $\rset_+^d\setminus\{0\}$, referred to as the exponent
 measure, 
 that is finite on sets bounded away from $0$ and such that
 \begin{equation}
   \label{eq:standard-rv}
   t\PP(\mb V \in tA  ) \xrightarrow[t\to \infty]{} \mu(A) \,, 
 \end{equation}
 for all set $A \in \borel(\rset^d)$ bounded away from $0$
 and such that $\mu(\partial A) = 0$. This is equivalent to vague
 convergence of the measures $\mu_t = t\PP(\mb V \in t\point)$ on the
 space $[0,\infty]^d\setminus\{0\}$~\citep{Resnick1987}, \citep{Resnick2007}
 and to $M_0$ convergence of the
 same collection of measures on $[0,\infty)^d\setminus\{0\}$ as later formalized
 in~\cite{hult2006regular} on
 a complete separable metric space. 
 An immediate consequence of~\eqref{eq:standard-rv} is that $\mu$ is homogeneous of order $-1$,
 $\mu(tA) = t^{-1} \mu(A)$ for $t>0$ and $A\in\borel(\rset^d)$. 
 Condition~\eqref{eq:standard-rv},  is a special case of \emph{regular variation} regarding the random vector $\mb V$: 
 a random vector $\mb Z$ is regularly varying  if there exists  a  real function $b(t)>0$ and a limit measure $\nu$, such that
 \begin{equation}
   \label{eq:nonstandard-rv}
   b(t)\PP(\mb Z \in tA  ) \xrightarrow[t\to \infty]{} \nu(A) \qquad (\nu(\delta A) =0 , 0\notin \cl( A ))
 \end{equation}
 where $b$ is a positive function such that  $b(tx)/b(t)\to x^{-\alpha}$ for all $x,t>0$. The exponent $\alpha$ is called the \emph{index of regular variation}.
 In the standard form~\eqref{eq:standard-rv} the normalizing function
 is $b(t) = t$ so that $\alpha=1$.  Thus
 condition~\eqref{eq:standard-rv} may seem overly stringent since it
 requires regular variation of $ \mb V$ in a standard form. However it is
 in fact weaker. 
 Indeed, assume that $\mb X$ satisfies only a nonstandard domain of
 attraction condition, namely that for multivariate sequences
 $\mb a_n = (a_n^{(1)}, \ldots, a_n^{(d)})$ with $a_n^{(j)}>0$ and
 $\mb b_n = (b_n^{(1)}, \ldots, b_n^{(d)})$ with $b_n^{(j)}\in\rset$, such
that $\PP(\mb X \le \mb b_n)\to 1$, the distribution
 $\law\big((\mb X - \mb b_n)/\mb a_n\given \mb X \not\le \mb b_n\big)$ converges to a
 nondegenerate limit.  Then $\mb X$ does not necessarily satisfy~\eqref{eq:nonstandard-rv}, however  $\mb V$, the standardized version of $\mb X$, automatically
 satisfies~\eqref{eq:standard-rv} (\cite{rootzen2006multivariate},
 Theorem 2.3 and \cite{Resnick1987}, Proposition 5.10).


 The exponent measure $\mu$ in the limit~\eqref{eq:standard-rv} may be
 viewed as the limit distribution of extremes, as
 $\law(t^{-1} \mb V\given \|\mb V\| \ge t)\wto c\mu(\point)_{|\ball^c}$ where we write
 $\ball^c = \rset_+^d\setminus\ball$, $\ball$ is the unit open ball in $\rset^d$ relative to some norm $\|\point\|$ on $\rset^d$ and $c = \mu(\ball^c)^{-1}$.  One
 characterization of $\mu$ relies
  on a transformation to polar
 coordinates: 
 for
 $\mb x\in [0,\infty)^d \setminus\{0\}$, set
 $\polar(\mb x) = (r(\mb x), \theta(\mb x))$ where $r(\mb x) = \|\mb x\|$ and
 $\theta(\mb x) = r(\mb x)^{-1} \mb x$ is a point on the positive orthant
 $\spherepos$ of the sphere, which we call the \emph{angle} (or direction) of
 $\mb x$. 
 Then the homogeneity
 property of $\mu$ implies that $\mu\circ \polar^{-1}$ is a product
 measure on $\rset_+^*\times \spherepos$, namely
 $\ud (\mu\circ \polar^{-1})(r,\mb w) = \frac{\ud r }{r^2} \otimes
 \ud \Phi (\mb w)$.  The angular component $\Phi$, usually called the
 \emph{angular measure}
 has finite mass and the above definition may be rephrased as follows:
 for all $t>0$ and Borel measurable $A\subset\sphere$, define the truncated cone with basis  $A$, 
 $
\cone_A = \{\mb x\in\rset_d: ~ r(\mb x)\ge 1, ~\theta(\mb x)\in A \}
 $. 
 The angular  measure of the  set $A$ is simply $\Phi(A) = \mu(\cone_A)$. By homogeneity, for all $t>0$, $\mu(t A) =
   t^{-1} \Phi(A)$. Finally, $\Phi$ is the limit distribution of the angle conditionally on  the norm being large,
 \begin{equation}
   \label{eq:angularMeasure}
   \law(\theta(\mb V)\given r(\mb V)>t)\wto c\; \Phi(\point), \qquad c = \Phi(\spherepos)^{-1} = \mu(\ball^c)^{-1}.
 \end{equation}
 Because the angular measure characterizes the exponent measure, a natural idea  for learning problems involving the limit distribution of extremes is to
 characterize optimal solutions in terms of $\Phi$ instead of $\mu$, and
 propose empirical solutions taking as input  extreme angles
 $\theta(\mb V_i)$'s such that $r(\mb V_i)$ is large, where $\mb V_i = v(\mb X_i)$ and $(\mb X_i)_{i\le n}\iidsim P$. This reduces the
 dimension of the sample space by one. This may seem little, but it should
 be  noticed that the removed radial dimension is the one along
 which the data points are likely to be the most spread out since the
 radial distribution behaves asymptotically as a power law, while the
 angular component is contained in the compact set $\spherepos$. Of course the marginal distributions are unkown, thus the Pareto transformation $v$ is also unkown but may typically be replaced with an empirical version. This
 line of thinking underpins the developments in the following sections concerning statistical learning for extreme data. 
 
\subsection{Statistical Learning Theory}\label{sec:LT-back}


Here we recall the fundamental concepts at the heart of the statistical explanation for the success of machine learning methods, and their ability to generalize well. The success of these predictive techniques can be illuminated by empirical process theory, quantification of the complexity of the function classes that index them, and concentration inequalities.
For an in depth introduction to statistical learning theory, refer \eg to \cite{lugosi2002pattern} or~\cite{bblStatLearnTheory}. For an excellent presentation of concentration inequalities, see \cite{BLM2013}. Essential monographs in statistical learning include \cite{Devroye96}, \cite{Vapnik00}, \cite{mohri2018foundations}, see also 
\cite{van1996weak} for empirical process theory, which intersects substantially with statistical learning. 
\smallskip

\noindent {\bf Empirical processes and Vapnik--Chervonenkis theory.}
Let $X, X_i, i\le n \iidsim P$ be a random element and \iid copies
valued in an arbitrary sample space  $\sspace$. Although in many applications $\sspace \subset \rset^d$, such a restriction is not required here and the theory covers potentially infinite-dimensional spaces, as it is the case in functional data analysis.  By
$P_n=n^{-1}\sum_{i=1}^n\delta_{\mb X_i}$ is meant the  empirical
distribution of the \iid sample $(\mb X_i, i\le n)$.
From a historical point of view, the framework developed by Vapnik and Chervonenkis provided a better understanding of predictive learning by studying fluctuations of  the empirical process $\{P_n(A):\; A\in \mathcal{A}\}$, where $\mathcal{A}$ is a class of (Borel measurable) subsets of $\sspace$. The \textsc{VC} shatter coefficient of class $\mathcal{A}$
\begin{equation}\label{eq:VC1}
  \shatter_{\class}(n)= \max_{(\mb x_1,\ldots, \mb x_n) \in \sspace^n} \left| \big\{  A \cap (\mb x_1, \ldots, \mb x_n) : A \in \class\big\}\right| 
\end{equation}
allows to obtain a distribution-free control of the (mean) uniform deviations of the empirical measure $P_n$, known as a \textsc{VC} inequality: with probability at least $1-\delta$,
 \begin{equation}\label{eq:vc-inProba}
   \sup_{A \in \class} |P - P_n|(A) \le  2\sqrt{ \frac{ 2 \log\big\{ 4 \shatter_{\class}(2n) /\delta\big\} }{n}}. 
 \end{equation}
  %
%

The combinatorial quantity $ \vapdim_{\class} = \sup\{ n\ge 1: \shatter_{\class}(n) = 2^n\}$ referred to as the \textsc{VC} dimension of $\class$ permits to bound $\log( \shatter_{\class}(n))$ when it is finite: by virtue of Sauer's lemma, we have $\shatter_{\class}(n)\leq (n+1)^{\vapdim_{\class}}$. 
  Many simple classes (\eg  half-spaces, hyperrectangles, ellipso\"ids in $\rset^d$, unions and intersections of such classes) have finite \textsc{VC} dimension, in particular classes of sets constructed by many popular classification algorithms (\eg  decision trees, neural nets, linear \textsc{SVM}). While the statistical literature makes greater use of metric entropies to quantify the complexity of function classes, the combinatorial approach can be likened to this, as explained in \eg  \cite{van2000asymptotic}.
%

\smallskip

\noindent {\bf Binary classification in the  \textsc{ERM} paradigm.}
The concepts briefly recalled above can be used to demonstrate the generalization ability of predictive rules learned by empirical risk minimization, in the case of binary classification in particular. A flagship problem in machine-learning, its study and algorithmic solutions serve as models for many other predictive learning problems, both supervised and unsupervised. Easy to formulate, it involves a random binary label $Y$ (the output), valued in $\{-1,+1\}$ say, as well as a r.v. $\mb X$ defined on the same probability space, taking its values in $\sspace$ 
and modelling some input information a priori useful to predict $Y$. The goal is to select a classifier $g:\sspace \to \{-1,\; +1\}$ in a class $\mathcal{G}$ with $0-1$
risk $R(g) = \PP(g(\mb X) \neq Y)$ nearly as small as the minimum risk over the ensemble of all classifiers, attained by the so called \emph{Bayes classifier} $g^*: \mb x\mapsto 2 \1\{\eta(\mb x) \ge 1/2 \} - 1$ where $\eta(\mb X)$
is the \emph{posterior probability}, $\eta(\mb x) = \PP(Y = 1 \given \mb X= \mb x)$. The joint distribution $P$ of the random pair $(\mb X,Y)$ being unknown, the selection in the supervised framework must be based on the observation of $n\geq 1$ training examples $(\mb X_1, Y_1), \ldots, (\mb X_n,Y_n)$, independent copies of $(\mb X,Y)$. The frequentist \ERM strategy, the main paradigm of machine-learning today, consists in trying to reproduce the available examples by minimizing a statistical version of the risk over $\mathcal{G}$, typically the counterpart of $R(g)$ obtained by replacing $P$ by the raw empirical distribution $P_n=(1/n)\sum_{i=1}^n \delta_{(\mb X_i,Y_i)}$ (or by a smoothed/convexified and/or additively penalized version of the latter)
\begin{equation}\label{emp_risk}
R_n(g) = \frac{1}{n}\sum_{i=1}^n \ind\{g(\mb X_i) \neq Y_i \} = P_n(A_g),
  \end{equation}
where $A_g= \{(\mb x,y)\in \sspace \times \{-1,+1\}:\; g(\mb x) \neq y \}$ for any $g\in \mathcal{G}$. The predictive performance of minimizers $g_n$ of~\eqref{emp_risk} over the class $\mathcal{G}$ is measured by the \textit{excess of risk} 
\begin{equation}\label{eq:risk_excess}
  R(g_n) - R(g^*)= P(A_{g_n})-P(A_{g^*}),
\end{equation}
the expected difference between the future prediction errors of $g_n$ and those of the optimum $g^*$ given the training data. Of course, no  usable analytic form exists for $g_n$ (a fortiori for $P(A_{g_n})$), because it is a function of the training examples which is the product of a complex optimization procedure (the error of which is neglected here). However~\eqref{eq:risk_excess} is classically bounded as follows:
 \begin{equation}\label{eq:classic_bound}
    R(g_n) - R(g^*) \le 2 \sup_{ g\in \cal G} \left| P(A_g) - P_n(A_g) \right| +\inf_{g\in \mathcal{G}}\{ R(g)-R(g^*)\}.
   \end{equation}
   While the second term on the right hand side of~\eqref{eq:classic_bound} measures the \textit{model bias} (and decreases as the class $\mathcal{G}$ gets larger), the first one (\textit{stochastic term}) can be controlled in expectation (or in probability) by means of inequality~\eqref{eq:VC1}. Under the assumption that the class of sets $\{A_g:\; g\in \mathcal{G}\}$ is of finite \textsc{VC} dimension (\ie  that $\mathcal{G}$ is a \textsc{VC} class of functions, see 2.6 in \cite{van1996weak}), the generalization capacity of classifiers learned using \textsc{ERM} can be assessed, the stochastic term being then of order $O_{\mathbb{P}}(1/\sqrt{n})$ up to a logarithmic factor. See  \cite{Devroye96} or \cite{Vapnik00} for a detailed presentation of statistical learning theory. The class $\mathcal{G}$ (\eg  of the hyperparameters of the learning algorithm), is chosen in order to nearly minimize the true risk or to approximately balance the two terms in~\eqref{eq:classic_bound}. This  is usually achieved  using data-driven methods for model selection, mainly cross-validation or additive penalization techniques in practice. These popular model selection procedures are analyzed in Sections \ref{subsec:CV} and \ref{sec:exlasso_section} in the context of `Learning on extremes'. 

Motivated by recent developments in the practice of machine-learning (\eg  nonlinear SVM, ensemble learning) in the last decades, alternative complexity assumptions (Rademacher averages, see \cite{CLV06} and the references therein) and additional results, related to tail bounds for maximal deviations in particular, have been recently elaborated to analyze machine-learning algorithms \citep{bblStatLearnTheory}.   Finally, the same type of tools can be used to guarantee the performance of the ERM principle for other supervised tasks, such as regression  \citep{LecMend} or ranking  \citep{CLV08}, and for unsupervised tasks as well, such as anomaly detection \citep{Scott2006} or clustering \citep{CLEMENCON201442}.

\section{Unsupervised Learning in the Tails}\label{sec:unsupervised}
We now show how to reconcile statistical learning and extreme value analysis, starting with the unsupervised framework. As explained below, this requires concentration results tailored to low-probability regions, involving variance in the bounds to account for data scarcity in the tails.

\subsection{Empirical Processes on Low-probability Regions}
\label{sec:lowProba_concentrationBounds}
One  main idea behind Peaks-Over-Threshold analysis in  \EVT is to retain the   $k$ largest order statistics ($k \ll n$) from a given sample to estimate tail distribution characteristics. In a multivariate setting, or more generally in a metric space equipped with scalar multiplication, data can still be ordered based on their norm or distance from the origin. This approach involves evaluating the empirical measure $P_n$ over a class of sets $\class = \{ tB : B \in \class_1 \}$, where $\class_1$ is any class of sets bounded away from the origin, and $t > 0$ is chosen such that the union of the class $\mathbb{A} = \cup_{A \in \class} A$ has a small probability $p = P(\mathbb{A}) = O(k/n)$. The classical empirical measure is then substituted with the tail empirical measure, defined as:
\[
\nu_k(A) = \frac{n}{k} P_n(A) = \frac{1}{k} \sum_{i=1}^n \un_A(\mb X_i),~~ A \in \class.
\]
Asymptotic results that have become standard in \EVT \citep[see \eg][]{mason1988strong,einmahl1992behavior,einmahl1992generalized,einmahl1997poisson}  indicate that under mild assumptions,  the tail empirical process $\sqrt{k}(\nu_k(A) - \nu(A))_{A \in\class}$ converges in some sense (weak convergence or strong approximation) to a nondegenerate process as $k\to\infty,k/n\to 0$.  It is therefore reasonable to anticipate concentration inequalities for the tail empirical measure, which can be expressed as:
$\textit{with probability $1- \delta$,  }
  \sup_{A \in \class} | \nu_k(A) - (n/k) P(A)| \le B_k(\delta), $
where $ B_k(\delta)$ is un upper bound resembling the VC bound~\eqref{eq:vc-inProba} v  with $n$ replaced with $k$.    Dividing both sides of the inequality  by $n/k$  and identifying $p$ and $k/n$, the desired result becomes: 
   \begin{equation*}
     \begin{gathered}
     \sup_{A \in \class} | P_n(A) -  P(A)|\le  O\left(  \sqrt{ \frac{ p \left\{\log(1/\delta) + \log(\shatter_\class(np) ) \right\}}{n}}\right)     .        
     \end{gathered}
   \end{equation*}
   The following  normalized VC-inequality (\cite{vapnik2015uniform,Anthony93}, see \cite{BBL05}, Section 5 for further discussions) comes close to this goal:  
    with probability $1 - \delta$,    
\begin{align*}
\sup_{A \in \class} \frac{P(A) - P_n(A) }{\sqrt{P(A)}}  ~~\le~~ 2 \sqrt{\frac{\log \shatter_{\class}(2n) +\log{\frac{4}{\delta}}}{n}}~,
\end{align*}
with a similar result regarding the supremum of $(P_n(A) - P(A) ) /\sqrt{P_n(A)}$. 
Notice that the upper bound in the above display  involves a logarithmic term $\log{\shatter_{\class}(2n)}$ depending on the total sample size, not the effective sample size $np$ as above. 

 The  VC-inequality stated below and proved in \cite{goix2015learning} achieves the goal stated above 
 and may be seen as  the cornerstone of several follow-up works at the intersection between \EVT and  statistical learning. 
If $\class$ is  a VC-class of sets  with VC-dimension $\vapdim_{\class}$, then  
 with probability at least $1-\delta$,
\begin{align}
\label{eq:concentration-colt}
\sup_{A \in \class} \left| P_n(A) - P(A)\right| ~~\le~~ C \bigg\{ \sqrt{p}\sqrt{\frac{\vapdim_{\class}}{n} \log(1/\delta) }+ \frac{1}{n} \log(1/\delta) \bigg\}~, 
\end{align}
where $C$ is a universal constant coming from chaining arguments. The fact that $C$ is not explicit  is arguably a weakness   in~\eqref{eq:concentration-colt}. A refined analysis in \cite{lhaut2021uniform} provides explicit constants, and several variants of the above results. Inspection of the constants and numerical experiments in the cited reference indicate that the best known bound seems to be, 
  \begin{align}
    \sup_{A \in \class} \left| P_n(A) - P(A)\right|
    &\le\sqrt{\frac{2p}{n}}\Big\{\sqrt{2\log(1/\delta)} +
      \sqrt{\log 2 + \vapdim_{\class}\log(2np+1)}  + \sqrt{2}/2\Big\} \nonumber\\
    & \qquad + \frac{2}{3n}\log(1/\delta)\,. \label{eq:best_lhaut}
  \end{align}
  The proofs of~\eqref{eq:concentration-colt} and~\eqref{eq:best_lhaut}  rely on  error decompositions involving the deviations of the $k${th}  order statistics from the theoretical $1-k/n$ quantiles, and a control of the deviations of an empirical risk, conditional  upon an excess above the latter quantiles. Classical arguments in statistical learning and empirical process theory such as symmetrization arguments
  (see \eg ~\cite{lugosi2002pattern}, \cite{bblStatLearnTheory}, \cite{BBL05}), combined with  concentration results leveraging the low variance of the Bernoulli variables $\un\{\mb X_i>t\}$ for large $t$ \citep[][Theorem 3.8]{McDiarmid98}, then lead to the  above results. 

  
  The above tail bounds are crucial for the nonasymptotic control of
  stochastic process fluctuations in multivariate \EVT, particularly
  for the empirical angular measure which is the focus of Section~\ref{sec:empiricalAngularMeasure} below.
  This control is essential for the
  statistical theory explaining the success of machine learning with
  extreme data. These bounds have also proven useful in other learning
  frameworks involving data scarcity, such as severely imbalanced
  classification, to provide guarantees for ERM algorithms minimizing
  a balanced risk \citep{aghbalou2023sharp}.

  As a concluding remark, while most applications of these principle have concerned the multivariate setting,  non-asymptotic deviation bounds  unrelated to VC Theory but based on specific concentration tools for tail order statistics,   have been instrumental  in univariate EVT settings for adaptive tail index estimation  \citep{boucheron2012concentration}, \citep{boucheron2015tail}, \citep{carpentier2015adaptive}, \citep{lederer2025adaptive} or bandit problems with tail sensitive rewards \citep{carpentier2014extreme}, \citep{achab2017max}.


  \subsection{Empirical Angular Measure}\label{sec:empiricalAngularMeasure}
  As recalled in Section~\ref{sec:background}, another  key
  characterization of multivariate tail dependence which fully leverages the homogeneity property of the limit measure $\mu$,   is the angular
  measure, which is traditionally defined as in~\eqref{eq:angularMeasure},
  as the limit angular distribution above high thresholds of
  \emph{marginally standardized} variables, with marginal standardization function $v:\rset^d\to[1,\infty)^d$ defined in~\eqref{eq:standardize-pareto}. 
In a realistic setting where the marginal distributions $F_j$ are unknown, an empirical rank transform is typically defined as 
  \begin{equation}
    \label{eq:ranktransform}
    \widehat v(\mb x)  = \Big(\frac{1}{1 - \widehat F_j(x_j)},~ j \in \dd \Big)\;~;\qquad 
    \widehat {\mb V }= \widehat v(\mb X), 
  \end{equation}
  where $\widehat F_j(x)  = (n+1)^{-1}\sum_{i\le n} \un\{X_{i,j}\le x\}, x\in\rset$, 
  and the empirical angular measure of a borel set $A\subset \spherepos$  is then
  $ \widehat \Phi(A) = k^{-1} \sum_{i\le n} \un\{ \widehat{\mb  V}_i \in (n/k)\,\cone_A\}$. 
  Working with angular regions complicates significantly the analysis of the error induced by marginal standardization,   compared with the rectangular regions involved in the analysis of  the standard tail dependence function. 
  Indeed, the marginal errors $\widehat F_j(x_j) - F_j(x_j)$ may not be analyzed separately from the deviations of the pseudo-empirical process involving the (unobserved) angles $\{\theta\circ v(\mb X_i): ~ i\le n\}$.  
Indeed the errors $\widehat F_j - F_j$ propagate in a nonlinear fashion onto the angular error of the rank transformed samples $\theta(\widehat{\mb V}_i) - \theta({\mb V}_i)$. The proof of asymptotic normality in the bivariate case  \citep{Einmahl2001}, \citep{Einmahl2009} relies heavily on  rewriting the empirical angular measure evaluated at $A\subset \sphere$ in terms of the empirical tail measure associated with pseudo-observations $V_i = v(\mb X_i)$, evaluated on a random set $\widehat \Gamma_A$ accounting from marginal randomness,   $\delta_{ \widehat V_i} (\cone_A) = \delta_{  V_i} (\widehat \Gamma_A)  $. The next step is to construct two deterministic framing sets $\Gamma_A^-$, $\Gamma_A^+$ such that $\Gamma_A^- \subset \widehat \Gamma_A \subset \Gamma_A^+$ with high probability. Due to nonlinearities, the expression for these framing sets is somewhat involved, whence the difficulty to extend the proof to the multivariate case.

 To our knowledge,  \cite{clemenccon2023concentration} is the first work establishing guarantees for the empirical angular measure going beyond consistency in arbitrary  dimension. These guarantees take the form of concentration inequalities for the supremum deviations, over a class of sets $\mathcal{A}$ composed of measurable subsets of the positive orthant $\spherepos$ of the sphere in $\rset^d$, relative to to the $\ell_p$ norm on $\rset^d$, $p\in[1,\infty]$.   From a technical viewpoint, a major advantage of the nonasymptotic
 approach is that it permits to construct framing sets similarly as  above that are not required to be `tight'. More precisely  the approximation error arising from such a framing in
 the error decomposition can be of the same order of magnitude as the
 other deviation  terms (\ie, with a leading term of order
 $\Oh(1/\sqrt{k})$), instead of being negligible compared to them,  as it is required in the asymptotic analysis of earlier works mentioned above.
 The class of framing sets considered  in \cite{clemenccon2023concentration} thus takes the (comparatively) simple form, 
 $$
 \Gamma = \Gamma^+(r,h)\cup \Gamma^-(r,h), \;
 \Gamma^{\sigma}(A) =
 \big\{
	\mb x \in [0, \infty)^d : \
	\|\mb x\|_p \ge \frac{1}{r}, \;
	\theta(\mb x) \in A^{\sigma}(h \|\mb x\|_p)
	\big\}, 
 $$
 where $\sigma\in\{+,-\}$, the numbers  $r> 1$ and $h>0$ are
 tolerance parameters which have explicit expressions,
 and $A^-(\eps), A^+(\eps)$ are respectively an inner and outer envelope of an angular set $A$.  Without going into details, framing sets are trumpet-shaped sets, with a gap between the target set and its framing sets increasing with the distance from the origin. This reflects the propagation of uncertainty in the empirical distribution functions $\widehat{F}_j(x_j)$ through the nonlinear transformations $1/\{1 - \widehat{F}_j(x_j)\}$. An illustration is provided in Figure~\ref{fig:framing-sets} 
 \begin{figure}[hbtp]
 \centering
 \includegraphics[width=0.2\linewidth]{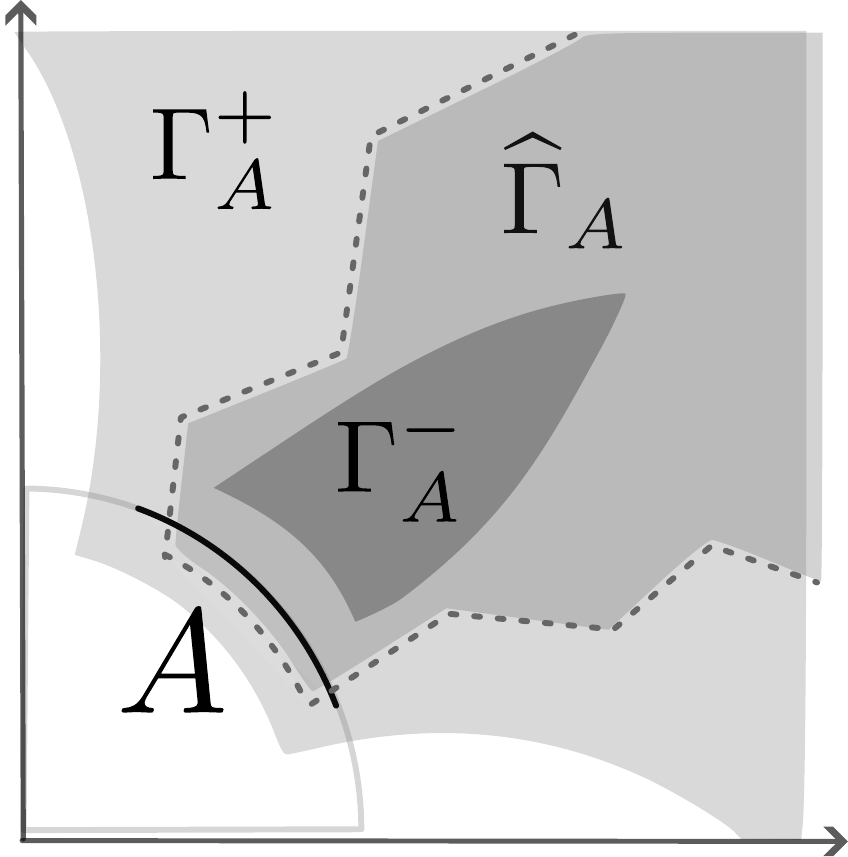}
 \caption{\label{fig:framing-sets} Bivariate illustration of the sets $A\subset\spherepos$, $\widehat{\Gamma}_A$, 
 and framing sets $\Gamma_A^- \subset \widehat{\Gamma}_A\subset \Gamma_A^+$ involved  in the theoretical analysis of the empirical angular measure.}
 \end{figure}
 
 Apart from (i) measurability and regularity conditions on the angular class $\mathcal{A}$, the main restrictions are that  (ii) the sets in the considered class are bounded away from the $2^d-1$ subfaces of $\spherepos$, and that (iii)
 the class $\Gamma$  of framing sets has finite VC dimension. 
Having to restrict the analysis to regions bounded away from the axes is no surprise in multivariate \EVT as regions close to the axes  constitute a recurrent issue that have motivated various censoring approaches  \citep{ledford1996statistics}. 
 The following bound is then valid with probability at least $1-\delta$  \citep[Theorem~3.1 in][]{clemenccon2023concentration}, 
 \begin{equation}
   \label{eq:main_results_empiricalAngular}
   \sup_{A\in\mathcal{A}}| \widehat \Phi(A) - \Phi(A) |  \le
   \frac{C_1(\delta, d, \vapdim_{\Gamma}, k)}{\sqrt{k}} +    \frac{C_2(\delta, d, \vapdim_{\Gamma},  k)}{k} +
   \textrm{Bias}(k,n). 
 \end{equation} 
 Here, $\mathrm{Bias}(k,n)$ is a bias term discussed below,  $\vapdim_{\Gamma}$ is the dimension of the class of framing sets, and  $C_1(\delta, d, \vapdim_{\Gamma}, k), C_2(\delta, d, \vapdim_{\Gamma}, k)$ are arguably complicated expressions, that are however explicit. In particular,  
 $C_1, C_2$  depend only logarithmically on $k, 1/\delta$, and polynomially on $d, \vapdim_{\Gamma}$, so that the bound does not become vacuous as   $d, \vapdim_{\Gamma}$ become large, as long as  the extreme sample size $k$ remains larger. The bias term $\textrm{Bias}(k,n)$ writes as $\sup\{ | (n/k)\PP(V \in G )  - \mu(G) |, G \in \Gamma\}$ and reflects the nonasymptotic nature of the largest observations. 
 It can typically be controlled by making additional second order assumptions, or in specific models. \cite{clemenccon2023concentration} work out a Bias-vanishing example in a multivariate Cauchy model. 

  From a broader perspective,
 such concentration results have unblocked several bottlenecks in the
 statistical learning approach of multivariate extremes, 
 in particular for anomaly
 detection based on angular \mv-set estimation
 \citep{thomas2017anomaly}, as described in Section~\ref{sec:mvsets}.  Other applications to supervised learning problems such as classification or regression on  extreme  covariates are reviewed in Section~\ref{sec:supervised}.


 Finally, similar results for an empirical estimator of an alternative
 characterization of the tail dependence structure, the standard tail
 dependence function namely, have been proved in
 \cite{goix2016sparse}. They have been leveraged in several further
 works \cite{goix2016sparse}, \cite{goix2017sparse} motivated by
 moderate-to-high dimensional contexts in where the goal is to
 identify subsets of components of a multivariate random vector which
 are likely to be simultaneously extreme, assuming that some sparse
 patterns exist, \ie that such subgroups are not too numerous and that
 their size is moderate. The latter sparsity assumption serves as a
 basis for a series of follow-up works with refined hidden regular
 variation assumptions \citep{simpson2020determining} or in a weakly
 sparse context
 \citep{chiapino2016feature}, \citep{chiapino2019identifying}. Concrete use
 cases in hydrology and aviation can be found respectively in
 \cite{chiapino2016feature} and \cite{chiapino2019vizu}. Other notable
 recent advances in unsupervised dimension reduction reduction for
 extremes with nonasymptotic guarantees include graphical Lasso
 approaches for learning tail conditional indepedence graphs
 \citep{engelke2021learning}. A nonasymptotic analysis of \ERM with
 similar guarantees on the tail statistical error is central to recent
 advancements in Principal Component Analysis (PCA) for multivariate
 extremes \citep{cooley2019decompositions}, \citep{drees2019principal}, with
 extensions to functional data analysis
 \citep{clemenccon2023regular}. This topic is the focus of a dedicated
 chapter in an upcoming edited volume.\footnote{\textit{Handbook on
     Statistics of Extremes}, Chapter 11,  co-authored by Dan Cooley,
    Anne Sabourin and Troy Wixson}

 \subsection{Unsupervised Anomaly Detection in Multivariate Tails via  Angular Minimum-Volume Set Estimation}\label{sec:mvsets}

This section illustrates the value of the general results of Sections~\ref{sec:lowProba_concentrationBounds} and~\ref{sec:empiricalAngularMeasure} in the unsupervised context of anomaly detection. The material presented here  follows \cite{thomas2017anomaly} and \cite{clemenccon2023concentration}. Minimum volume sets (\mv-sets in short), extending univariate quantiles, are the smallest sample space subsets containing  at least $\alpha$ probability mass, at some  level $\alpha$ \citep{einmahl1992generalized}.  This  approach shares similarity with \eg \cite{cai2011estimation}, where estimation of low levels of the density function using multivariate \EVT is also considered in a somewhat different context, that is assuming joint regular variation with a single regular variation index as in~\eqref{eq:nonstandard-rv}. In the cited reference, consistency of the extreme level sets is established. Here  a different approach is taken by assuming regular variation of the standardized vector $V$ and working with preliminary standardized data. Importantly,  nonasymptotic upper bounds concerning the estimated level sets are obtained  in \citeauthor{thomas2017anomaly} and ~\citeauthor{clemenccon2023concentration}. The statistical analysis in \citeauthor{thomas2017anomaly} is limited to the ideal case where the
marginal distributions are known, while  the work
of~\citeauthor{clemenccon2023concentration} on the empirical angular
measure, which encompasses rank transformation, provides the missing
piece to adress this limitation.

As detailed below, \mv-sets are strong candidates for regions of the samples space labelled as `normal' (\ie not abnormal) in an anomaly detection framework. With this in mind, \citeauthor{thomas2017anomaly} propose an anomaly detection algorithm aimed at detecting anomalies \emph{among extremes}, \ie within tail regions of the sample space of the kind  $\{\mb x \in \rset^d: \|\mb x\| > t \}$ for large values of $t$, under regular variation assumptions.
The envisioned setting here is moderate dimensional, meaning  that one may assume that the angular measure of extremes is concentrated on the interior of the positive orthant of the unit sphere.  
Higher dimensional settings are the focus of alternative algorithms based on dimension reduction, as discussed at the end of Section~\ref{sec:empiricalAngularMeasure}.

Minimum-volume sets (\(\mv\)-sets) are fundamental to semi-supervised
anomaly detection, where the goal is to define a `normal' region using
only majority-class data. In a Neyman–Pearson framework, a point is
flagged as abnormal at level \(1 - \alpha\) if it lies outside the
\(\mv\)-set of level \(\alpha\)~\citep{blanchard2010semi}. Under
absolute continuity of \(P\) with respect to \(\lambda\) and
boundedness of the density \(f = \ud P / \ud \lambda\), the set
\(\Omega^*_{\alpha} = \{\mb x \in \sspace: f(\mb x) \geq F_f^{-1}(1 -
\alpha)\}\) uniquely solves the minimum-volume problem
\(\min_\Omega \lambda(\Omega)\) subject to
\(\PP(\Omega) \geq \alpha\)~\citep{Polonik97}. Estimation involves
optimizing over a subclass \(\class\) of controlled complexity,
replacing \(P\) with its empirical counterpart. Nonasymptotic
guarantees and algorithmic approaches for empirical \(\mv\)-sets
\(\widehat{\Omega}_\alpha\) are provided
in~\cite{Scott2006}, \cite{einmahl1992generalized}.

Denoting by $\lambda$ the Lebesgue measure on
$\spherepos$, 
the optimization problem solved in \citeauthor{thomas2017anomaly} to
produce an empirical angular \mv-set $\widehat{\Omega}_{\alpha}$ on
the positive orhtant $\spherepos$ of the sphere is
  \begin{equation}\label{eq:mvsetsphere_emp}
  \min_{\Omega \in \class}\lambda(\Omega) \text{ subject to } \hatphi(\Omega)\geq \alpha-\psi(\delta) \, .
\end{equation}
where $\psi(\delta)$ is a tolerance parameter which magnitude should be of the same order as the deviations of the empirical measure $\hatphi$ described in Section~\ref{sec:empiricalAngularMeasure}.  
As for the choice of the $\ell_p$ norm involved in the definition of $\Phi$,  $p =\infty$ turns out to be a convenient choice from a computational perspective in \citeauthor{thomas2017anomaly}, although the theory  developed in \citeauthor{clemenccon2023concentration} (see Section~\ref{sec:empiricalAngularMeasure}) allows for an arbitrary choice of $p\in[1,\infty]$.  As summarized in \citeauthor{clemenccon2023concentration}, as soon as $\psi(\delta)$ is set to a value at least as large as the right-hand side of the error bound~\eqref{eq:main_results_empiricalAngular}, then on the favourable event $ \mathcal{E}$ of probability greater than $1-\delta$, over which  the error bound holds, it also holds that
\begin{equation}
  \label{eq:guarantee_mvset}
  \begin{aligned}
    \Phi_p(\widehat{\Omega}_{\alpha}) & ~\ge ~  {\hatphi}_p(\widehat{\Omega}_{\alpha}) - \psi \ge \alpha - 2\psi\,,\qquad \text{and } \\
    \lambda(\widehat{\Omega}_{\alpha}) & ~ \le ~  \inf \left\{ \lambda(A) : \ A \in \class, \, \Phi_p(A) \ge \alpha \right\}.
  \end{aligned}
\end{equation}

	
In the suggested  framework,  \emph{extreme} data are observed values $X$ such that the norm of their standardization  $r(\mb V) = \Vert \mb V \Vert$ is large. 
Among extreme observations with comparable   (large) radius,  \emph{anomalies}  are those which \emph{direction} $ \theta(\mb V) = \mb V  / \|\mb V \|$ is unusual, which is an appropriate model for anomalies in many applications. 

These results can be extended naturally, especially since, in anomaly detection (within multivariate tail regions), the objective is often to \emph{rank} suspicious observations based on their degree of abnormality rather than simply classifying as 'abnormal' \textit{vs} 'normal'.
One way of achieving this could be to combine the previous results with the approach developed in \cite{CT18}. One may also refer to \citeauthor{thomas2017anomaly}, where an ad-hoc scoring function for extremes is proposed, in the form of the product of a radial
component, $s_r(\mb x) = 1/r^2(\widehat{v}(\mb x)) $ and an angular component
$ \widehat{s}_\theta(\theta(\widehat{v}(\mb x))) $ derived from angular
minimum volume sets.

\section{Supervised Learning on Extreme-Valued Covariates}\label{sec:supervised}
We now turn to supervised learning problems, aimed at predicting the (discrete or  continuous) labels $Y$ assigned to extreme input observations $\mb X$ (covariates). We show how the multivariate regular variation hypothesis makes it possible to define a notion of limit risk reflecting the prediction error conditional on an extreme value of the covariates 
and to establish generalization bounds for the minimizers of an empirical version of the latter in classification and regression. Model selection via cross-validation is analyzed in this context. We also show here that this assumption allows us to design useful data representation and augmentation methods in the context of word embedding, the cornerstone of modern natural language processing.

\subsection{Binary Classification on the Tails of the Covariates}\label{sec:classif_theory}


The material presented here relies mainly on  
\cite{jalalzai2018binary} and~\cite{clemenccon2023concentration} -- the latter extending the guarantees obtained in the former, to the case where marginal distributions are unkown. 

Classification is the flagship of supervised learning problems. It is
also one of the  most natural frameworks in which  the uniform concentration
bounds  introduced as background in
Section~\ref{sec:background} reveal themselves fruitful for
proving generalization guarantees of classifiers obtained {via}
\ERM. 
Consider a
classification problem where a random pair $(X,Y)$ is observed, where
$X$ is an explanatory variable and $Y\in \{-1, +1\}$ is the label to
be predicted.  Suppose that the goal is to predict the labels
associated to large explanatory variables say $\|X\|\ge t$ for some
large threshold $t$. 
Two scenarios are possible: (i) one class becomes predominant as the threshold $t \to \infty$, making the problem almost trivial; (ii) the distribution of positive and negative classes stabilizes and tends toward a limit as $t \to \infty$. Our interest lies in case (ii).
Before formalizing the \ERM approach proposed in~\cite{jalalzai2018binary},  the following example,  drawn from~\cite{aghbalou2022cross}, illustrates a plausible situation corresponding to case (ii), in connection with the multivariate regular variation setting.
\begin{example}[Prediction in regularly varying random vectors]\label{ex:classif_rv}
  This example is  taken from~\cite{aghbalou2022cross}. 
  Consider the task of
  predicting the occurrence of an extreme event, namely predicting the
  (missing) value of a component \( Z_{d+1} \) in a random vector
  \( \mb Z = (Z_1, \ldots, Z_{d+1}) \in \mathbb{R}^{d+1} \), based on the
  partial observation \( (Z_1, \ldots, Z_d) \), given that the latter
  is large. An intermediate problem could be to predict whether
  \( Z_{d+1} \) is also large. Define:
$$
\mb X = (Z_1, \ldots, Z_d) \quad \text{and} \quad Y = \un\left\{ \frac{Z_{d+1}}{\|\mb Z\|} > c \right\},
$$
where \(\|\cdot\|\) is the \(\ell^p\) norm for some \( p \in [1, \infty) \) and \( c \in (0,1) \) depends on the task. For instance, \( c =\{(1/(d+1)\}^{1/p} \) if the target event is \( Z_{d+1} > 0 \) and \(|Z_{d+1}|^p\) is at least as large as the average value of \(|Z_j|^p\) for \( j \leq d+1 \).
\cite{aghbalou2022cross}  prove (Appendix A.2 of the reference) that the pair \((X, Y)\) satisfies the requirements of~\cite{jalalzai2018binary}'s setting if \( Z \) is a heavy-tailed random vector with a regularly varying density, an assumption commonly used in \EVT~\citep{de1987regular}, \citep{cai2011estimation}.  
\end{example}
Classification by ERM involves selecting a classifier \( g_n \) from a class \( \mathcal{G} \) to minimize an empirical risk. However, focusing on errors above a threshold \( t \) presents challenges: classical ERM may not perform well in the tails due to negligible training error influence, and restricting the training set to tail regions may result in insufficient data for generalization.
The primary goal of~\cite{jalalzai2018binary} is to minimize the conditional error probability for excesses above a radial threshold as \( t \to \infty \):
 \begin{equation}\label{eq:extreme-risk}
R_t(g) := \PP\{Y \neq g(\mb X) \mid \|\mb X\| > t\},
 \end{equation}
where \( P_t \) is the conditional distribution of \( (\mb X, Y) \) given \( \|\mb X\| > t \). The risk at infinity is defined as:
 \begin{equation}
   \label{eq:riskInfinity}
R_\infty(g) = \limsup_{t \to \infty} R_t(g).
 \end{equation}
 The Bayes classifier \( g^* \) relative to $P$  minimizes \( R_\infty \), but there is no guarantee that the ERM classifier \( g_n \) performs well in the tail, especially if \( \mathcal{G} \) is parametric, due to negligible tail errors compared to bulk errors.
 To address data scarcity in the tails, it is  assumed that the class distributions \(\PP(\mb X \in \cdot \mid Y = \sigma)\), \(\sigma \in \{-1, +1\}\), are regularly varying. Additionally, the ratio \(\PP(Y = +1 \mid \|\mb X\| > t) / \PP(Y = -1 \mid \|\mb X\| > t)\) must converge to a finite, nonzero limit to ensure the problem is neither trivial nor insoluble. This implies the indices of regular variation are equal, \(\alpha_+ = \alpha_-\). The tail index is set to 1 for simplicity and the normalizing functions are set to  \(b_+(t) = b_-(t) = t\), as if the explanatory variable were marginally standardized, see~\eqref{eq:standard-rv}; however, inspection of the proof of Theorem 1 in the cited reference shows that the choice of normalizing functions \(b_+(t)\) and \(b_-(t)\) does not affect the results as long as \(b_+(t) / b_-(t) \to \ell \in (0, \infty)\). This leads to the assumption that for all \(\sigma \in \{-, +\}\), the conditional distribution of \(\mb X\) given \(Y = \sigma\) is regularly varying with limit measure \(\mu_{\sigma}\) and angular measure \(\Phi_{\sigma}\), 
\begin{equation}
  \label{eq:simplifiedRV_jalalzai}
t \PP(t^{-1} \mb X \in A \mid Y = \sigma) \xrightarrow[t \to \infty]{} \mu_{\sigma}(A), \quad \sigma \in \{-, +\},  
\end{equation}
for measurable \(A \subset [0, \infty)^d \setminus \{0\}\) with \(0 \notin \partial A\) and \(\mu(\partial A) \neq 0\).
 A limiting pair \((\mb X_\infty, Y_\infty)\) is defined by the distribution
\[
\PP(Y_\infty = 1) = p_\infty, \quad \PP(\mb X_\infty \in A \mid Y_\infty = y) = \frac{\mu_{\text{sign}(y)}(A)}{\Phi_{\text{sign}(y)}(\spherepos)}.
\]
The regression function \(\eta_\infty(\mb x) = \PP(Y_\infty = 1 \mid  \mb X_\infty =\mb  x)\) depends only on the angle \(\theta( \mb x)\). 
Under the aforementioned assumptions, 
it turns out~\citep[Theorem 1 in ][]{jalalzai2018binary} that the  Bayes classifier \(g_{P_\infty}^*\) relative to the (standard) risk for the limit pair, that is, the minimizer of  $g\mapsto \PP\{g(\mb X_\infty)\neq Y_\infty\}$ also  minimizes the limit risk $R_\infty(g)$ defined in~\eqref{eq:riskInfinity}, 
\[
\inf_{g \text{ measurable}} R_\infty(g) = R_\infty(g_{P_\infty}^*) = \EE\Big[\min\big\{\eta_\infty(\Theta_\infty), 1 - \eta_\infty(\Theta_\infty) \big\} \Big], 
\]
where $\Theta_\infty = \theta(\mb X_\infty)$ is the limit  angular variable. 
An immediate consequence is that the  optimal classifier for \(R_\infty\) depends solely on the angle \(\theta(x)\) of the explanatory variable, suggesting an ERM strategy focused on angular classifiers. The  straightforward approach  subsequently analyzed is to  minimize, over a class \(\mathcal{G}\) of predictors $g(x)$ depending solely on the  angle $\theta(\mb x)$ with finite VC-dimension $\vapdim$,  an empirical risk \(\widehat{R}_k(g) = k^{-1}\sum_{i\le k}\un\{Y_{(i)} \neq  g(\mb X_{(i)})\}\)
where $[(1),\ldots, (n)]$ is the permutation associated with the (nonincreasing) order statistics of the norms $\|\mb X_i\|$, namely $\|\mb X_{(1)}\|\ge \dotsb\ge \|\mb X_{(n)}\|$. In~\cite{jalalzai2018binary} (Theorem 2),   an upper bound is derived on the excess risk above finite levels  $R_{t(n,k)}(\widehat g) - R_{t(n,k)}^*$, where $\widehat g$ is the minimizer of $\widehat{R}_k$, $t(n,k)$ is the $1-k/n$ quantile of  $r(\mb X)$  and $R_{t(n,k)}^*$ is the infimum of the risk over the class  \(\mathcal{G}\).   Under the assumptions listed above the upper bound is of order $ \sqrt{\vapdim}
( \sqrt{\log(1/\delta) /k} + \log(1/\delta)/k )$, up to a bias term  $\inf_{g\in \mathcal{G}_{\sphere}}R_{t}(g)-R^*_{t}$ reflecting that the optimal classifier at level $t$ may not belong to the model. 

We now turn to a more realistic setting where the marginal
distributions of the covariate may not all be on the same scale, and
in particular, may not share a common tail index, so that  the covariate $\mb X$ may not
satisfy~\eqref{eq:nonstandard-rv}. A natural idea is then to use the
transformed variable $\mb V= v(\mb X)$, and assume that~\eqref{eq:simplifiedRV_jalalzai} holds only for the transformed pair $(\mb V,Y)$. In practice the algorithm would
take as input the rank-transformed variables
$\widehat{\mb V}_i = \widehat v(\mb X_i)$. The analysis conducted in~\cite{clemenccon2023concentration} of the rank
transformation \(\widehat{v}\) defined in~\eqref{eq:ranktransform}, precisely allows to control the deviations of the empirical risk in this setting. 
The key observation linking the empirical
estimation of the angular measure to the classification problem is
that evaluating the empirical risk of a classifier \(g\) on extreme
covariates involves counting the positive (resp. negative) instances
\( \{ \theta(\widehat{\mb  V}_i), Y_i = +1 \text{ resp. } -1 \}\) such that
\(\|\widehat{\mb  V}_i\| \ge \|\widehat{\mb  V}_{(k)}\| \), observed in positively
(resp. negatively) assigned regions
\(\sphere_{+1}(g) = \{ \mb s \in \sphere: g(\mb s) = +1\}\)
(resp. \(\sphere_{-1}(g) = \{\mb s \in \sphere: g(\mb s) = -1\}\)). In other
words, the empirical risk of \(g\) is fully characterized by the
empirical angular measures of the positive and negative classes.

Controlling the deviations of the empirical angular measures of each class, while accounting for the rank transformation, permits to  control the excess risk of a variant of the ERM strategy described above, where the input \(X\) is replaced with its rank-transformed version \(\widehat{V}\). The deviations of the empirical risk are then controlled by restricting the input space to regions sufficiently far from the axes. Introducing, for \(\tau > 0\), the class of subsets of the sphere
\[
\mathcal{A} = \{\sphere_{+1}(g) \cap \{\mb s: \min \mb s \ge \tau\}, g \in \mathcal{G}\} \cup \{\sphere_{-}(g) \cap \{\mb s: \min \mb s \ge \tau\}, g \in \mathcal{G}\},
\]
and making the same assumptions regarding the class \(\mathcal{A}\) as those summarized in Section~\ref{sec:empiricalAngularMeasure} before the statement of the bound~\eqref{eq:main_results_empiricalAngular}, they obtain
\[
\sup_{g \in \mathcal{G}} |\widehat{R}^{>\tau}(g) - R_\infty^\tau(g)| \le \frac{C_1(\delta/2, d, \vapdim_{\bar{\mathcal{A}}}, k)}{\sqrt{k}} + \frac{C_2(\delta/2, d, \vapdim_{\bar{\mathcal{A}}}, k)}{k} + \textrm{Bias\,II}(k, n),
\]
where \(\widehat{R}^{>\tau}\) and \(R_\infty^\tau\) are restrictions of the empirical risk and the asymptotic risk to inputs \(\mb x\) such that \(\min \theta(\widehat{v}(\mb x)) > \tau\) and \(\min \theta(v(\mb x)) > \tau\), respectively. The functions \(C_1\) and \(C_2\) are as in~\eqref{eq:main_results_empiricalAngular}, and \(\textrm{Bias\,II}( k, n)\) is a bias term of the same nature as in~\eqref{eq:main_results_empiricalAngular}, with class distributions \(\PP(\mb V \in \cdot, Y = \sigma 1) \) and their associated limit measures \(\mu_\sigma\) replacing the distribution \( \PP(\mb V \in \cdot) \) of the covariate and its limit angular measure \(\mu\).

\subsection{Heavy-Tailed Representations and Classification in NLP}\label{sec:NLP}
This section summarizes the findings of~\cite{jalalzai2020heavy}, which adapts the classification framework from Section~\ref{sec:classif_theory} to a natural language processing task. While embeddings such as  BERT (which was state-of-the-art at the time of this work's publication) excel at capturing statistical regularities in text data, they overlook the heavy-tailed nature of word frequency distributions~\citep{mandelbrot1953informational}, \citep{extrem_3}, \citep{extrem_2}. The proposed method, \lhtr (Learning a Heavy Tailed Representation), transforms BERT outputs $\mb X$ into representations $\mb Z = \varphi(\mb X) \in \rset^{d^\prime}$ that satisfy \EVT assumptions (namely, regular variation), even when the original embeddings do not. This transformation is learned adversarially~\citep{Goodfellow-et-al-2016} to ensure that:
\begin{itemize}
\item The marginal distribution $q(\mb z)$ of $\mb Z$ approximates a heavy-tailed target distribution $p$ with $b(t) = t$ (see~\eqref{eq:nonstandard-rv}).
\item The classification loss of a downstream multilayer perceptron trained on $\mb Z$ is minimized.
\end{itemize}
The extreme region is defined as $\{\mb z: \|\mb z\| > t\}$, where $t$ is an empirical quantile of the encoded data norms. \lhtr trains two classifiers: $g^{\text{ext}}$ for the extreme region and $g^{\text{bulk}}$ for the bulk, combining them as:
$$
g(\mb z) = g^{\text{ext}}(\mb z)\ind\{\|\mb z\| > t\} + g^{\text{bulk}}(\mb z)\ind\{\|\mb z\| \leq t\}.
$$
The method minimizes a weighted risk:
\begin{multline*}
  R(\varphi, g^{\text{ext}}, g^{\text{bulk}}) = \rho_1 \PP(Y \neq g^{\text{ext}}(\mb Z), \|\mb Z\| \geq t) + \rho_2 \PP(Y \neq g^{\text{bulk}}(\mb Z), \|\mb Z\| < t) \\
  + \rho_3 \mathfrak{D}(q(\mb z), p(\mb z)),
  \end{multline*}
where $\mathfrak{D}$ is the Jensen--Shannon distance, approximated adversarially.

Experiments on the \textit{Amazon}~\citep{mcauley2013hidden} and \textit{Yelp}~\citep{yelp_1}, \citep{yelp_2} datasets demonstrate that \lhtr outperforms baseline models, particularly in tail regions. Additionally,  a label-preserving data augmentation algorithm leveraging \lhtr is proposed. Synthetic sequences are generated along the curve $\{\varphi^{-1}[\lambda \varphi(\mb x)] : \lambda > 1\}$ for a given input $\mb x$, preserving the original label and improving classification performance. The target distribution is a multivariate logistic distribution, $F(\mb x) = \exp\big\{ - (\sum_{j=1}^d {x_j}^{1/\delta})^{\delta} \big\}$ with $\delta\in (0,1]$. Evaluations show consistent improvements in classification error, F1 score, and qualitative metrics for data augmentation tasks.

 \subsection{Cross-validation Guarantees}\label{subsec:CV}

 Cross-validation (CV) is a widely used tool in statistical learning for estimating the generalization risk of algorithms and selecting hyper-parameters or models~\citep{arlot2010survey}, \citep{wager2020cross}, \citep{bates2021cross}. While CV's performance has been analyzed in various settings, including density estimation~\citep{arlot2008VFpen}, \citep{arlot2016VFchoice} and least-squares regression~\citep{homrighausen2013lasso}, \citep{xu2020rademacher}, there is a lack of theoretical guarantees for CV when applied to \EVT-based algorithms.

 The work by~\cite{aghbalou2022cross} aims to address the gap in the literature by examining learning algorithms based on \ERM  in low-probability regions of the covariate space, as explored in~\cite{jalalzai2018binary}, \cite{jalalzai2020heavy}, \cite{clemenccon2023concentration} for classification tasks described in Sections~\ref{sec:classif_theory} and~\ref{sec:NLP},   and in~\cite{huet2023regression} for continuous regression settings (Sections~\ref{sec:regression},~\ref{sec:exlasso_section} below). A broader class of problems where cross-validation comes as a natural approach include unseupervised contexts, \eg for goodness-of-fit evaluation or model selction in parametric modelic of tail dependence~\citep{Einmahl2012}, \citep{einmahl2018continuous}, \citep{einmahl2016m}, \citep{kiriliouk2019peaks} and model selection. For dimension reduction in multivariate extremes, CV could be used for selecting hyper-parameters and sparsity levels  in~\cite{goix2016sparse}, \cite{goix2017sparse} and PCA-based methods~\citep{cooley2019decompositions}, \citep{jiang2020principal}, \citep{drees2019principal}. Clustering approaches~\citep{janssen2019k}, \citep{chiapino2019vizu}, \citep{jalalzai21feature} could also benefit from effective model selection techniques. In supervised settings, extreme quantile regression is well-established in \EVT, with notable contributions from~\cite{daouia2013kernel}, \cite{chernozhukov2017extremal}, \cite{chavez2014extreme}, and CV could be used for kernel bandwidth selection. Recent alternatives like Gradient Boosting~\citep{velthoen2021gradient}, Regression Trees~\citep{farkas2021cyber}, and Extremal Random Forests~ \citep{gnecco2024extremal} 

Within the vast landscape of potential applications described above,~\cite{aghbalou2022cross} focuses on the \ERM classification framework developed in~\cite{jalalzai2018binary} for moderate-to-high dimensional contexts. They consider a constrained form of the Lasso:
\begin{equation}
  \label{eq:constrained_lasso_cv}
\minimize_{\boldsymbol{\beta} \in \mathbb{R}^d} \sum_{i \leq k} c(g_{\boldsymbol{\beta}}(\mb X_{(i)}), Y_{(i)}) \quad \text{subject to} \quad \|\boldsymbol{\beta}\|_1 \leq u,  
\end{equation}
where \(u > 0\) is a hyper-parameter to be selected by CV, \(g_{\boldsymbol{\beta}}(\mb x) =  \boldsymbol{\beta}^\top \theta(\mb x)\) aligns with the theoretical results in~\cite{jalalzai2018binary} that classification on large covariates should depend only on their angle, and \(c\) is the logistic cost, \(c(\hat{y}, y) = \log\{1 + \exp(\hat{y} y) \}\), a convex substitute for the 0-1 loss.
Their analysis is designed to handle more general settings of risk minimization on a low-probability region of the sample space, specifically for \ERM machine learning algorithms minimizing empirical versions of the risk:
\[
\risk(g, \mb Z) = \EE[c(g, \mb Z) | \|\mb Z\| > t_p],
\]
where \(\mb Z\) is a random observation in a sample space \(\mathcal{Z}\), \(\|\cdot\|\) is a semi-norm on \(\mathcal{Z}\), and \(t_p\) is the \(1-p\) quantile of \(\|\mb Z\|\). Thus, \(\mathbb{A} = \{\mb z \in \mathcal{Z}: \|\mb z\| > t_p\}\) is an unknown region of the sample space with low probability \(p = \PP( \mb Z \in \mathbb{A} ) \ll 1\), aligning with the `rare events' setting developed in Section~\ref{sec:lowProba_concentrationBounds}.

Given training data \( \mb Z_1, \ldots, \mb Z_n\) and a training subsample \((\mb Z_i, i \in \mathcal{S})\) indexed by \(\mathcal{S} \subset \{1, \ldots, n\}\), an empirical version of the risk \(\risk\) is:
\[
\riskhat(g, \mathcal{S}) = \frac{1}{p n_{\mathcal{S}}} \sum_{i \in \mathcal{S}} c(g, \mb Z_i) \ind\{\|\mb Z_i\| > \|\mb Z_{(\lfloor p n \rfloor)}\|\}.
\]
The focus of~\cite{aghbalou2022cross} is on learning rules \(\Psi\) that take \(\mathcal{S}\) as input and return the \ERM solution \(\Psi(\mathcal{S}) = \hat{g}(\mathcal{S}) = \argmin_{g \in \mathcal{G}} \riskhat(g, \mathcal{S})\). The main quantity of interest is the generalization risk \(\risk(\hat{g}_n)\) of the \ERM predictor \(\hat{g}_n = \Psi(\{1, \ldots, n\})\) trained on the full dataset. A CV estimator of this quantity is defined as an average of hold-out estimates:
\[
\rcvEx(\Psi, V_{1:K}) = \frac{1}{K} \sum_{j=1}^{K} \riskhat(\Psi(T_j), V_j),
\]
where \((V_j, j \leq K)\) are validation sets and \(T_j = \{1, \ldots, n\} \setminus V_j\) are training sets.

Besides a balance condition for the validation and training sets, satisfied in common CV schemes including leave-one-out, leave-p-out, and K-fold, their standard working assumptions are: (i) the loss class \(\{\mb z \mapsto c(g, \mb z), g \in \mathcal{G}\}\) associated with the predictor class \(\mathcal{G}\) is a VC subgraph class ~\citep[see \eg][Section 2.6]{van1996weak} and (ii) the cost function is bounded. It should be noted that the boundedness assumption precludes application to extreme quantile regression, leaving the extension to unbounded losses with appropriate tail control as an open question for further work.

The main results take the form on upper  bounds on the error $|\rcvEx(\Psi, V_{1:K})  - \risk(\hatg_n)|$, valid with high probability.
First an \emph{exponential} error bound is derived\footnote{The denomination `exponential'   comes from the fact that the probability upper bound results from an inversion of a  tail bound of exponential form,
  $\PP\big( \text{error} > s \big)
  \le \exp\{ - f(s) \}$ where
  $f(t)\ge \min(A_1 t, A_2 t^2)  $ for some context-dependent factors $A_1,A_2$.} (Theorem 3.1 in \cite{aghbalou2022cross}),
 \begin{align}
   \big|\rcvEx(\Psi,V_{1:K})-
   \risk\big(\hatg_n\big)\big|  \leq  E_{\text{CV}}(n_{T},n_{V},p) 
                                                     + \frac{20}{3 n p} \log(1/\delta) 
   +  20\sqrt{\frac{2}{np} \log(1/\delta)},\label{eq:expo_bound_cv}
\end{align}
valid with probability $1-15\delta$, where 
 $ E_{\text{CV}}(n_{T},n_{V},p)=M\sqrt{\mathcal{V}_\mathcal{G}}(1/\sqrt{n_{V}p}+4/\sqrt{n_{T}p})+5/(n_{T}p)$, 
where $M >0$ is a universal constant, and $n_V$ and $n_T$ are the respective size of the validation and training sets.  For $K$-fold CV schemes, both $n_V$ and $n_T$ are proportional to $n$ and the above bound ensures in particular consistency as the sample size grows while the number of folds is fixed. However for leave-one-out schemes and variants, the size $n_V$ of the validation set is small, and error bounds invoving a term $1/(n_V p)$ are inappropriate in particular when $p$ is small. Different techniques of proof permit anyway to obtain a  \emph{polynomial} error bound\footnote{`Polynomial' means it results from an inversion of a  tail bound of (inverse) polynomial  form, 
  $\PP\big( \text{error} > s \ \le  B_1 + B_2/t$ up to negligible terms,   for some context-dependent factors $B_1,B_2$.} involving the training sample size $n_T$ only,
\begin{align}
  \big|\rcvEx(\Psi,V_{1:K})-
  \risk\big(\hatg_n\big)\big| \leq 
  E'_{\text{CV}}(n_{T},\alpha)+
  \frac{1}{\delta\sqrt{n_{T}\alpha}}(5M\sqrt{\mathcal{V}_{\mathcal{G}}}+M_5),
  \label{eq:poly_bound_cv}
\end{align}
with probability $17\delta$,         where $M,M'>0$ are universal constants, $M$ is the same as
        in~\eqref{eq:expo_bound_cv} and
        $E'_{\text{CV}}(n_{T},\alpha)=
        9M\sqrt{\mathcal{V}_\mathcal{G}}/\sqrt{\alpha
            n_{T}}+9/(n_{T}\alpha)$.
        Both bounds~\eqref{eq:expo_bound_cv} and~\eqref{eq:poly_bound_cv} serve as sanity-check guarantees that do not prove the CV error outperforms the naive (biased) method of substituting the empirical training risk for the generalization risk. This limitation mirrors that in~\cite{cornec10}, \cite{cornec17}, established outside the \EVT setting. As discussed in~\cite{aghbalou2022cross}, moving beyond sanity-check guarantees without additional assumptions remains an open question in mathematical statistics.
        
        Returning to the constrained Lasso problem, a grid search over a range $U$  of plausible values for \(u\) necessitates a union-bound approach to control the deviations of the CV risk for the rules $\Psi_u$ associated with problem~\eqref{eq:constrained_lasso_cv} with constraint $\|\boldsymbol{\beta}\|_1\le u$, $u\in U$. This would render the polynomial bound~\eqref{eq:poly_bound_cv} vacuous. However, the exponential bound~\eqref{eq:expo_bound_cv} remains effective because multiplying \(\delta\) by the size \(|U|\) of the \(u\)-grid results in only an additional logarithmic factor, \(\log|U|\). Lemma 5.1 in~\cite{aghbalou2022cross}  states an upper bound valid with high probability $1- 15\delta$, 
        \begin{multline*}
\big|\rcvEx(\Psi_{\widehat u},V_{1:K})-
\risk\big(\hatg_n\big)\big| \le 
\max(U)\bigg[ 
     2E(n,K, p)
     + \frac{40}{3 n p} \log\left(|U|/\delta)\right) \\
            +  40\sqrt{\frac{2}{n p} \log\left(|U|/\delta\right)}\bigg],         
        \end{multline*}
where $\widehat u$  is the minimizer of the CV risks $\rcvEx(\Psi_u V_{1:K}), u \in U$, and $E(n,K, p)=  5M\sqrt{(d+1)K/(n p)}  +5K/\{ (K-1)n p \}$.  
\medskip

\noindent {\bf Discussion.} The work~\cite{aghbalou2022cross} is an initial theoretical attempt to provide guarantees for CV in  \EVT problems. As mentioned earlier, the condition of a bounded loss may be seen as restrictive, although it aligns well with the 'learning on extreme covariates' setting. This setting encompasses prediction problems in multivariate regularly varying random vectors, as seen in Example~\ref{ex:classif_rv}, or in the prediction setting developed in Proposition~\ref{prop:exampleRescale} below. The main topic not covered is extreme quantile regression, as discussed previously. In another direction, one could consider moving away from the ERM context and explore stable algorithms~\cite{kearns1999algorithmic}, \cite{bousquet2002stability}, \cite{kutin2012}, \cite{kumar2013near}, which encompass a wide range of algorithms, including stochastic gradient descent strategies and regularized risk minimization approaches.

\subsection{Regression on Covariate Tails}\label{sec:regression}

Having established a robust classification framework for large covariates, this section extends the approach of~\cite{jalalzai2018binary} to regression. This transition is nontrivial, as it requires adaptations to handle continuous outcomes, based on the material in~\cite{huet2023regression}, namely to the task of predicting through least square regression a continuous target $Y\in\rset$, based on a covariate
$\mb X\in\rset^d$, conditional to the occurrence of an extreme event relative to the covariate, $\| \mb X\|>t$, for  $t\gg 1$.  
The focus of~\cite{huet2023regression} is on the minimization of  conditional least-squares risk,
$R_t(f) = \EE(Y-f(\mb X)\,|\, \|\mb X\|>t )$, and its limit superior as $t\to \infty$,
$R_\infty(f) = \limsup_{t\to\infty}R_t(f) $, where $f$ is a prediction function chosen in an appropriate class $\mathcal{F}$ of predictors.
We state in this section the foundational assumptions that underpin the penalized extension developed in Section~\ref{sec:exlasso_section}.  
\begin{assumption}[Bounded target]\label{assum_1_huet}
  The target $Y$ is a.s. bounded, \ie $Y\in[-M,M]$ with probability $1$, for some $M>0$. 
\end{assumption}
Although Assumption~\ref{assum_1_huet} may seem restrictive in an \EVT context, it is important to note that, similar to classification settings, the `extreme' behavior considered here pertains to the covariate \(\mb X\), not the target \(Y\). Additionally,~\cite{huet2023regression} provides an illustrative example involving multivariate random vectors and a prediction task of one component based on the others, paralleling the classification illustration in Example~\ref{ex:classif_rv}. In this example, a scaling mechanism constructs an appropriate target \(Y\) that satisfies the boundedness assumption. Furthermore, in Proposition~\ref{prop:exampleRescale}, we present a new example that simplifies the former. Specifically, the boundedness assumption translates into a condition that the angular component of the target should be bounded away from the axes. This should not be surprising given the restrictions imposed regarding regions near the axes, as discussed in Sections~\ref{sec:empiricalAngularMeasure} and~\ref{sec:classif_theory}.
We now state an equivalent from Assumption~2 in~\cite{huet2023regression}, which proves to be more convenient for our purposes. 
\begin{assumption}[Regular variation w.r.t. the covariate]\label{assum_2_huet} 
  The function $t\mapsto \PP(\|\mb X\|>t)$ is regularly varying with index
  $\alpha>0$, \ie $\PP(\|\mb X\|>ts ) / \PP(\|\mb X\|>t)\to s^{-\alpha}$ for
  all $s>0$, as $t\to\infty$, and
  $$
  \mathcal{L}\big\{ (t^{-1} \mb X, Y)|~  \|\mb X\|>t \big\}
  \xrightarrow[t\to\infty]{} P_\infty 
  $$
  for some limit distribution $P_\infty$ on $\{(\mb x,y)\in\rset^{d+1}: \|\mb x\|> 1, y \in \rset \}$ 
\end{assumption}
Automatically, under Assumption~\ref{assum_2_huet}, the limit distribution $P_\infty$ is $\alpha$-homogeneous \wrt its first component,  $P_\infty(t A, B) = t^{-\alpha}P_\infty(A,B)$. Importantly, denoting by $(\mb X_\infty, Y_\infty)$ a random pair distributed according to $P_\infty$, the homogeneity of $P_\infty$ implies that $\|\mb X_\infty\| \perp\!\!\!\perp (\theta(\mb X_\infty), Y_\infty)$. A major consequence is that
the regression function $f_{P_\infty}$   for the limit pair $(\mb X_\infty,Y_\infty)$, defined by $f_{P_\infty}(\mb X_\infty) =  \EE(Y_\infty ~|~ \mb X_\infty)$ almost surely,  does not depend on the radial component $r(\mb X_\infty)$. In other words,   there exists a function  $h_\infty$ defind on $\sphere$ such that
$f_{P_\infty}(\mb x) = h_\infty(\theta(\mb x))$.
The next step is to establish optimality properties for $f_{P_\infty}$ regarding $R_\infty$, paralleling the ones in the classification setting. An additional condition is needed, regarding the convergence of the Bayes regression function defined almost surely by  $f^*(\mb X) = \EE(Y | \mb X)$, towards $f_{P_\infty}$.  
\begin{assumption}\label{hyp:cont_reg_func}
  The  regression function $\bayesext$ for the limit pair $(\mb X_\infty,Y_\infty)$  is continuous on $\rset^d\setminus\{0_{\mb{R}^d}\}$ and  as $t$ tends to infinity,
  \begin{equation*}
    \EE\Big( \,\big|\,f^*(\mb X) - f_{P_\infty}(\mb X)\,\big| ~\Big|~ \lVert \mb X\rVert \ge  t \,\Big) \xrightarrow[t\to\infty]{}0.
  \end{equation*}
\end{assumption}
Several concrete examples are provided where Assumption~\ref{hyp:cont_reg_func} is satisfied, including multiplicative and additive noise models, and an example involving regular variation (\wrt the covariate) of densities.

Under Assumptions~\ref{assum_1_huet},~\ref{assum_2_huet}, and~\ref{hyp:cont_reg_func},~\cite{huet2023regression} establish that \(f_{P_\infty} = h_\infty \circ \theta\) is indeed a minimizer of \(R_\infty\). This suggests an \ERM strategy paralleling the classification setup, namely searching for an angular predictor of the form \(f = h \circ \theta\). This involves choosing a predictor class of the form \(\mathcal{F} = \{h \circ \theta \,:\, h \in \mathcal{H}\}\) and minimizing the empirical tail risk:
\[
\min_{h \in \mathcal{H}} \frac{1}{k} \sum_{i \leq k} \big\{Y_{(i)} - h \circ \theta(\mb X_{(i)}) \big\}^2.
\]
Under standard measurability and VC complexity assumptions regarding the class \(\mathcal{H}\), for any $\delta\in (0,1)$, we have with probability \(1 - \delta\) at least (Proposition 3.1 in \cite{huet2023regression}):
\begin{multline}
  \sup_{h\in \mathcal{H}}
  \left| \widehat{R}_{k}(h \circ \theta) - R_{t(n,k)}(h \circ \theta) \right|  \le V_k =
  4 M^2\bigg(~
    \frac{2 \sqrt{ 2 \log( 3 / \delta)} + C \sqrt{V_{\mathcal{H}}}}{
    \sqrt{k}}  \\
    +~ \frac{  \frac{4}{3} \log( 3/\delta ) +  V_{\mathcal{H}}}{k} ~\bigg), \label{eq:deviation_bound_huet}
\end{multline}
where \(C\) is a universal constant and \(V_{\mathcal{H}}\) is the VC-subgraph dimension of \(\mathcal{H}\). This `deviation bound' can naturally be combined with the following decomposition of the excess risk at infinity (Theorem 3.3 in the cited reference)

\begin{align}\label{eq:excess}
 R_\infty(\widehat f)- R_\infty^* \le D_k + B_1(t_{n,k}) + B_2(\mathcal{H}), 
\end{align}
where $D_k = 2 V_k$ --see \eqref{eq:deviation_bound_huet}-- is a deviation term, and $ B_1,B_2$ are  two bias terms, 
\begin{equation*}
  \begin{cases}
    B_1(t) =  2\sup_{h \in \mathcal{H}} |R_\infty(h\circ\theta)- R_t(h\circ\theta)| & \text{(threshold bias)} \\
    B_2(\mathcal{H}) = \inf_{h \in\mathcal{H}} R_\infty(h \circ\theta) \;-\; R_\infty(h_\infty\circ\theta) &\text{(class bias)}. 
  \end{cases}
\end{equation*}
The first bias term, $B_1$, is specific to \EVT and arises from the non-asymptotic behavior of observations exceeding a fixed threshold $t$. It can be shown to vanish under additional regularity conditions on the class $\mathcal{H}$, such as total boundedness with respect to the uniform norm, or uniform boundedness of both the predictors and the angular densities; see Proposition 3.2 of \citeauthor{huet2023regression}. The class bias term, $B_2$, captures the potential discrepancy between the class of predictors and the regression function relative to the limiting distribution. This type of bias is ubiquitous in learning theory and is not specific to \EVT. 


\medskip

\noindent{\bf Illustrating the scope of assumptions.} We provide an intuitive, new example that fits the setting of Assumptions~\ref{assum_1_huet}, \ref{assum_2_huet}, and~\ref{hyp:cont_reg_func}. This can be viewed as a simplification of Example 2.2 in \citeauthor{huet2023regression}, which involves a nonlinear rescaling of the possibly unbounded target.
In the following statement, let $\|\point\|_{\rset^d}$ and
$\|\point\|_{\rset^{d+1}}$ be the $\ell_p$ norm,  $p \in
[1,\infty]$,  in $\rset^d$ and
$\rset^{d+1}$ respectively.\footnote{The argument is valid for any norms
$\|\point\|_{\rset^d},
\|\point\|_{\rset^{d+1}}$ such that canonical basis vectors have norm
equal to one, and such that $\|(\mb x,0)\|_{\rset^{d+1}} =
\|\mb x\|_{\rset^d}$ for $\mb x \in \rset ^d$.} For simplicity when clear from the context, the subscripts $\rset^d$, $\rset^{d+1}$ are dropped. 
\begin{proposition}[Rescaling an unbounded target and enforcing
  Assumptions~\ref{assum_1_huet},~\ref{assum_2_huet} and~\ref{hyp:cont_reg_func}]\label{prop:exampleRescale} ~
 \begin{enumerate} 
 \item  Let  $(\mb X,Z)$ be  regularly varying in $\rset^{d+1}$,
with limit distribution 
$$\Pi_\infty = \lim_{t\to\infty}\mathcal{L}\big(t^{-1}(\mb X,Z) ~|~\|(\mb X,Z)\|>t \big)$$ such that $\Pi_\infty\{(\mb x,z): \|\mb x\|>1\}>0$, and let $Y = Z / \|\mb X\| $. 
    Then the  random pair $(\mb X,Y)$ satisfies Assumption~\ref{assum_2_huet}, with limit distribution 
    $$
    P_\infty(\point) = \frac{\Pi_\infty\circ\varphi^{-1}(\point)}{
    \Pi_\infty\big\{(\mb x,z)~:~ \|\mb x\|>1\big\}
    },
    $$
    where  the mapping $\varphi:~(\rset^d \setminus\{0\}) \times \rset \to   (\rset^d \setminus\{0\}) \times \rset $ is a homeomorphism defined by  $\varphi(\mb x, z) = (\mb x,  z / \|\mb x\|)$. 
  

  \item Conversely, let  $(\mb X,Y)$ be a random pair satisfying 
    Assumptions~\ref{assum_1_huet}
    and~\ref{assum_2_huet}, with a limit distribution $P_\infty$, 
  Define $Z = \|\mb X\|Y$. Then the random pair $(\mb X,Z)$ is regularly varying in $\rset^{d+1}$, with limit distribuution $\Pi_\infty = P_\infty\circ \phi$. 
\item
  Let  $(\mb X,Z)$ be a random pair having continuous, regularly varying density $\pi$ on $\rset^{d+1}\setminus\{0\}$, with limit density $\pi_\infty$,  \ie there exists  a regularly varying function $b(t)$ with index $\alpha>0$ such that
  \[
\sup_{\|(\mb x,z) \| >1} \big| t^{d+1}b(t) \pi(t \mb x,tz)  - \pi_\infty(\mb x,z) \big|\xrightarrow[t\to\infty]{} 0. 
\]
Assume in addition that  the rescaled variable $Y=Z/\| \mb X\|$ is bounded (Assumption~\ref{assum_1_huet}).   Then the rescaled pair $( \mb X,Y)$ satisfies Assumption~\ref{assum_2_huet}, and it  has a continuous, regularly varying density
$p (\mb x,y) =  \|\mb x\| \pi( \mb x,\| \mb x\| y)$,  with limit density
$ p_\infty ( \mb x,y)= \| \mb x\| \pi_\infty(\mb x, \|\mb x\|y)$, and same scaling function as that of the pair $(\mb X,Z)$,
namely
\begin{equation}
  \label{eq:reg_var_density_xy}
  \sup_{\| \mb x\|\ge 1, y \in\rset}
  \big| 
  b(t) t^{d} p(t \mb x,y) - p_\infty(\mb x,y) \big| \xrightarrow[t\to\infty]{} 0. 
\end{equation}
Finally,   under the additional condition that the limit marginal density
$\pi_{\mb x,\infty}(\mb x) = \int_\rset \pi_\infty(\mb x,y)\ud y$ is lower bounded on
$\sphere$ in $\rset^d$,  \ie $\inf_\sphere \pi_{\mb x,\infty}(\omega)>0$ then the pair $(\mb X, Y)$ also satisfies Assumption~\ref{hyp:cont_reg_func}. 
\end{enumerate}
\end{proposition}
The proof is given in Appendix~\ref{ap:sectionRegression}.
  A convenient feature of the setting envisioned in Proposition~\ref{prop:exampleRescale},   is that guarantees regarding a prediction function $\widehat h\circ \theta(\mb x)$ relative to $Y := Z/\|\mb X\|$ immediately yield guarantees on the (rescaled)  error of the predictor $\widehat Z := \|\mb X\| \,\widehat h\circ \theta(\mb X)$.  
  Indeed   the (squared) scaled error then writes
  $
( \widehat Z - Z)^2 =  \|\mb X\|^2 ~(\widehat Y  - Y)^2.
  $
  
\section{High-Dimensional Extreme Covariates: \\ XLASSO}\label{sec:exlasso_section}

\subsection{Framework and Preliminaries}
The framework examined in \cite{huet2023regression} is intentionally simplified: the proposed algorithms minimize an empirical version of the squared error risk without incorporating a penalization term. This approach becomes impractical in common scenarios where the class of predictors is complex. A quintessential example of such a scenario is linear regression, where the feature space is \(\mathcal{X} = \mathbb{R}^d\) and the class of candidate prediction functions is $$
\mathcal{H} = \{h_{\boldsymbol{\beta}}: \mb x \mapsto \langle \boldsymbol{\beta}
, \mb x \rangle, 
\, \boldsymbol{\beta} \in \mathbb{R}^{d}\}.$$
In this case, \(\mathcal{H}\) is a VC-subgraph class with a VC-dimension \(V = d + 1\) \citep[\eg,][Chapter 3]{anthony2009neural}. When \(d\) is comparable to the (extreme) sample size,  overfitting becomes a significant issue. This exemplifies the limitations of traditional statistical methods and has motivated the development of high-dimensional statistics.

A prominent algorithm for high-dimensional settings, especially when
the optimal predictor is sparse, is the celebrated Lasso (Least
Absolute Shrinkage and Selection Operator) introduced in
\cite{tibshirani1996regression}. Lasso offers provable guarantees in
these scenarios, making it a cornerstone in high-dimensional
statistics and machine learning. For a pedagogical presentation of
theoretical results on Lasso, see Chapter 11 in
\cite{hastie2015statistical}. This section aims to extend some of these
results (namely, a bound on the prediction error) to least squares regression on extreme covariates and
demonstrate that mainstream theoretical results in high-dimensional
statistics can be adapted to the framework of \EVT under appropriate
and interpretable assumptions, with minimal additional complexity.

We  name XLASSO the learning algorithm defined by the following penalized risk minimization problem, which is a natural extension of the (Lagrangian) Lasso setting, 
with  same class $\mathcal{H}$ of linear predictors as above.
As in previous sections let  $[(i), 1\le i\le n]$ denote a random permutation such that
$\|\mb X_{(1)}\|\ge \cdots\ge \|\mb X_{(n)}\|$ and let $k\ll n$ and $\lambda>0$ be fixed. Then XLASSO solves the following convex optimisation problem, 
\begin{equation}
  \label{eq:lasso_problem}
  \operatorname{minimize}_{\boldsymbol{\beta}\in \rset^{d}}~~ \frac{1}{2 k}\sum_{i=1}^k \big\{Y_{(i)} - h_{\boldsymbol{\beta}}\circ\theta(\mb X_{(i)})\big\}^2 + \lambda \,\|\boldsymbol{\beta}\|_1~.  
\end{equation}
We let $\widehat{\boldsymbol{\beta}}$ denote the solution of~\eqref{eq:lasso_problem}. The form of~\eqref{eq:lasso_problem} is identical to that  of the standard Lasso, which allows the use of any standard machine learning library to solve it in a reasonable amount of  computational time.  
Although similar to the standard Lasso problem from a computational perspective, the theoretical  analysis of the performance of the solution of~\eqref{eq:lasso_problem}  requires some care, insofar as 
extracting  the subsample of variables associated with the $k$ largest norms $\|\mb X_{(i)}\|$ breaks the independence property of the original sample. In addition, because our interest is on the tails of the covariates, some work is needed regarding model assumptions ensuring some statistical guarantees for the solution of \eqref{eq:lasso_problem}. 
\begin{remark}[Lagrangian versus constrained Lasso]
 For simplicity, we focus our presentation on the Lagrangian Lasso estimator, which is the solution to Equation~\eqref{eq:lasso_problem}. Straightforward extensions to the constrained Lasso can be readily derived using similar, albeit simpler, arguments. For instance, in \cite{aghbalou2022cross}, a constrained logistic-Lasso algorithm is examined as a primary example of a model selection problem involving extreme covariates, and the choice of the constraint level is addressed via cross-validation.
\end{remark}

\subsection{Asymptotic Linear Model on Extreme Covariates}
Our primarily assumption is that a linear relationship exists between the target and the angular component of the predictor at asymptotic levels. 
The fact that the linear relationship concerns the \emph{angular component} of $ \mb X$ as $\|\mb X\|$ goes to infinity, is in line with the general theory of regression on extreme covariates developed in \cite{huet2023regression}.   We consider an heteroscedastic model to facilitate further developments in the context of  unbounded targets. 
\begin{assumption}[Linear model on extreme covariates]\label{as:linear}
  For some  $\boldsymbol{\beta}^*\in\rset^{d}$, 
$$Y =  \theta(\mb X)^\top \boldsymbol{\beta}^* + b(\mb X) +  \sigma(\mb X)  \varepsilon,  $$ 
where $\varepsilon$ is a  bounded, centered noise independent from $\mb X$,
$|\varepsilon|\le 1 $ 
almost surely,  the noise variance $\sigma(\mb x)>0$ satisfies:
$$
\begin{aligned}
  M_\varepsilon = \sup_{\mb x} \sigma(\mb x)<\infty. 
\end{aligned}
$$
Also, for some continuous angular function $\sigma_\theta$ defined on the unit sphere and bounded by $M_{\varepsilon}$,
$$
\sup_{\|\mb x\|>t} |\sigma(\mb x) - \sigma_\theta(\theta(\mb x))| \xrightarrow[t \to\infty]{} 0.
$$
In addition
the bias function 
$b:\rset^d\to \rset$ is bounded and vanishes at infinity,
$$
\sup_{ \|\mb x\|>t} |b(\mb x)|\xrightarrow[t\to\infty]{} 0.
$$
\end{assumption}
 The above  assumption of a linear relationship between $Y$ and $\theta(\mb X)$ only holds asymptotically as $\|\mb X\|\to \infty$, and the inclusion of a bias term $b(\mb X)$ must be accounted for in some way in the analysis.
For simplicity, we do not account for an offset term, although such an extension can easily be achieved at the cost of moderate additional notational complexity. 
If $\|\point\|$ is the  $\ell_1$ norm, the covariates are related by an  affine relation $\theta_{d}(\mb x) = 1 - \sum_{j=1}^{d-1}\theta_j(\mb x)$, so that including all  $d$ components of $\theta(\mb X)$ in the model is equivalent to including an offset term (\ie a constant predictor).   Another reasonable model would be to use any norm,
remove one component, say $\theta_d(\mb x)$, from the family of predictors, while including  an offset term. However this would break the symmetry among covariates and complicate notations, and we do not pursue this idea further.  Finally, the assumption that the bias term vanishes at infinity, so that the limit pair $(\mb X_\infty,Y_\infty)$ follows an exact linear model, could be weakened at the cost of additional technicality and error terms, leveraging related ideas in \cite{buhlmann2011statistics}. 



It turns out that the asymptotic linear model in
Assumption~\ref{as:linear} is a specific case of a generic noise
model considered in \citep[][Example 2.1]{huet2023regression}, 
namely $Y = g(\mb X, \varepsilon)$ where for all $\varepsilon$,
$\sup_{\|\mb x\|>t} | g(\mb x,\varepsilon) - g_\theta(\theta(\mb x), \varepsilon) |\to 0$, for
some bounded, continuous function $g_\theta$ defined on $\sphere$.
The following lemma derives immediately from this observation, and
shows that the main results in \cite{huet2023regression} apply.
\begin{lemma}[Assumption~\ref{as:linear} and \cite{huet2023regression}'s framework]\label{lem:link_linear_huet}
  Let $(\mb X,Y)$ satisfy
  the asymptotic linear model in Assumption~\ref{as:linear}, and assume that $\mb X$ is regularly varying. Then $(\mb X,Y)$ satisfies
  Assumption~\ref{assum_1_huet}, \ref{assum_2_huet}
  and~\ref{hyp:cont_reg_func}, with
  $M \le \|\boldsymbol{\beta}^*\|_\infty + \|b\|_\infty + M_\varepsilon$.
\end{lemma}

The following key example  leverages Proposition~\ref{prop:exampleRescale}  and demonstrates that the bounded target assumption on \(Y\)  (or equivalently, on $M_\varepsilon$) does not disqualify unbounded targets, which are frequent in Extreme Value Theory, provided an appropriate rescaling 
is applied. Informally, the assumption is that for large $\|\mb X\|$,  and some \emph{homogeneous} function $s(\mb x) = \|\mb x\| s_\theta(\theta(\mb x))$, 
$$
Z ~\approx~  \left\langle \mb X,  \boldsymbol{\beta} \right\rangle ~+~ s(\mb X) \varepsilon  ~+~  o(\,\|\mb X\|\,)\,. 
$$
\begin{example}[Rescaling an unbounded target -- prediction in a regularly varying random vector]\label{eg:rvVector}
  Let $(\mb X,Z)\in\rset^d\times \rset$ be a random pair (covariate, target)  where $\mb X$ is regularly varying, $\mb X\neq 0$ almost surely,  and assume the following semi-parametric linear model
  \begin{equation}
    \label{eq:linear_model_RV}
    Z ~=~ \left\langle \mb X , \boldsymbol{\beta}^* \right\rangle ~+~  B(\mb X) ~+~ s( \mb X ) \varepsilon\,,     
  \end{equation}
  where the bias  function $B$ satisfies  
  $$\sup_{\mb x \in\rset^d} \frac{|B(\mb x)| }{\|\mb x\|\vee 1} = M_B<\infty, \textrm{ and } \sup_{\|\mb x\|>t} \frac{|B(\mb x)| }{\|\mb x\| } \xrightarrow[t\to\infty]{} 0; $$
  the noise $\varepsilon$ is  centered,  bounded by $1$ and independent of $X$ as in Assumption~\ref{as:linear}, 
and the variance function satisfies
$$
\sup_{\mb x \in\rset^d}\frac{s(\mb x)}{\|\mb x\|\vee 1} = M_\varepsilon<\infty, 
$$
and 
$$
\sup_{\|\mb x\|>t} \Big|\frac{s(\mb x)}{\|\mb x\|} - \sigma_\theta(\theta(\mb x)) \Big|\xrightarrow[t\to\infty]{} 0, 
$$
where $\sigma_\theta$ is a continuous function defined on the sphere $\sphere$.

  Then, letting $Y = Z/(\|\mb X\|\vee 1)$, the following statements hold true 
  \begin{enumerate}
  \item $(\mb X,Y)$ satisfies Assumption~\ref{as:linear} with
    $$ b(\mb x) = B(\mb x)/(\|\mb x\|\vee 1), ~~
    \sigma(x) = s(x)/(\|\mb x\|\vee 1)  $$  
  \item  $(\mb X,Y)$ satisfies Assumption~\ref{assum_1_huet} with
    $M = M_\varepsilon + M_B +\|\boldsymbol{\beta}^*\|_\infty$,  as well as  Assumptions~\ref{assum_2_huet} and~\ref{hyp:cont_reg_func}
  \item $(\mb X,Z)$ is regularly varying. 
\end{enumerate}
Statement $1$ above derives immediately from the conditions encapsulated respectively in~\eqref{eq:linear_model_RV} and in Assumption~\ref{as:linear}. 
Statement $2$ is a consequence of Statement~$1$ combined with Lemma~\ref{lem:link_linear_huet}. Finally Statement $3$ is a direct  application of Proposition~\ref{prop:exampleRescale}. 
  \end{example}

  \subsection{XLASSO: Statistical Guarantees}

 We now present some  nonasymptotic  statistical results regarding the prediction error in the framework of a tail linear model described in Assumption~\ref{as:linear}. We first   introduce convenient notations, 
 $\mb y   = (Y_{(1)},\ldots, Y_{(k)} )\in\rset^k$ is the vector of observed targets associated with the $k$ largest covariates,
 $$\mb W = (\theta( \mb X_{(1)})^\top,\ldots, \theta(\mb X_{(k)}) ^\top )^\top \in \rset^{k\times p
 }$$
 is the design matrix made of the angular components of these covariates. The residual vector ${\mb e} = \mb y  - \mb W\boldsymbol{\beta}^*$ shall play a crucial role in the analysis. 


 We can now  reformulate Theorem~1.2-a 
in  \cite{hastie2015statistical} (originally proved in \cite{bunea2007sparsity}) in our framework, regarding
the prediction error $\|\mb W (\widehat{\boldsymbol{\beta}} - \boldsymbol{\beta}^*) \|_2$. 
The skeptical reader may doubt the applicability of such a result in our context which departs significantly from the standard linear regression problem, as detailed above. However it should be noted that the  result borrowed from \cite{hastie2015statistical} is valid \emph{with probability one}, as it relies  solely on algebraic manipulations and on the optimality property of $\widehat{\boldsymbol{\beta}}$ with respect to \eqref{eq:lasso_problem}.

\begin{lemma}[Prediction  error,  \cite{bunea2007sparsity}, Theorem  11.2-a in \cite{hastie2015statistical}]\label{lem:estim-error-largeLambda}
  Assume that the penalty term is chosen sufficiently large, namely  $\lambda \ge 2 k^{-1} \|\mb W^\top {\mb e}\|_\infty $. The (in-sample) prediction error of the XLASSO estimator then satisfies 
  \begin{equation}
    \label{eq:predError_slow_largeLambda}
 k^{-1}\|\mb W( \widehat{\boldsymbol{\beta}} - \boldsymbol{\beta}^*)\|_2^2  \le 12 \,\lambda \,\|\boldsymbol{\beta}^*\|_1 .       \end{equation}
  
\end{lemma}

\begin{remark}\label{rem:restrictedEV_fastrates_estimation}
  Considering a learning problem on extreme covariates aims to account
  for the high variability of the input, making a fixed design setting
  inappropriate. Thus, we avoid additional assumptions like
  'restricted eigenvalue conditions' or 'irrepresentability
  conditions' on the design matrix $\mb W$, which could control the
  estimation error $\|\widehat{\boldsymbol{\beta}} - \boldsymbol{\beta}^*\|_2$ and achieve fast
  rates on the prediction error
  $\|\mb W (\widehat{\boldsymbol{\beta}} - \boldsymbol{\beta}^*) \|_2^2$ (see Chapter 11 in
  \cite{hastie2015statistical}). Instead, we focus on establishing
  slow rates for the prediction error without special conditions on
  the design matrix. We conjecture that fast rates and control over
  the estimation error could be achieved under appropriate assumptions
  on the tail distribution of $\theta(\mb X)$, by adapting arguments from
  \cite{rudelson2012reconstruction}. This question is left for future
  work.
 \end{remark}
To establish upper bounds on the prediction errors $\mb W(\widehat{\boldsymbol{\beta}} -\boldsymbol{\beta}^*)$, control over the residual of extremes $2 k^{-1}\|\mb W^\top {\mb e}\|_\infty$ is required. This control allows for the subsequent establishment that the penalty $\lambda$ in Lemma~\ref{lem:estim-error-largeLambda} can be chosen sufficiently small to yield a nonvacuous bound. This key step in the analysis diverges from existing approaches surveyed in~\cite{hastie2015statistical}. 
As discussed above, the main technical bottleneck is to control the deviations of the residuals ${\mb e} = \mb y - \mb W\boldsymbol{\beta}^*$. 

\begin{proposition}[Deviations of the residual vector]\label{prop:deviations_residual} Let Assumption~\ref{as:linear} hold true,  
  assume that $\|\mb X\|$ has a continuous distribution. 
  With probability $1-\delta$,
  $$
k^{-1}\|\mb W^\top {\mb e}\|_\infty  \le M_\varepsilon\sqrt{\frac{\log(4d/\delta)}{2k}} + \bar b(t_{n, \tilde k(\delta/2)}),  
$$
where $t_{n,\kappa}$ denotes the $1-\kappa/n$ quantile of the random variable $\|\mb X\|$,  $$\tilde k(\delta)  = k\Big(1 + \sqrt{\frac{3\log(1/\delta)}{k}} +  \frac{3\log(1/\delta)}{k} \Big), $$ 
and  $\bar b(t) = \sup_{\|\mb x\|>t}b(\mb x)$.
\end{proposition}

Our main result derives immediately from Lemma~\ref{lem:estim-error-largeLambda} and Proposition~\ref{prop:deviations_residual}.
\begin{theorem}[XLASSO:  prediction guarantees]\label{theo:exlasso}
  Let Assumption \ref{as:linear} hold true and let $\|\mb X\|$ have a continuous distribution.  
  Define the tolerance level
  $$B(k,\delta) = M_\varepsilon\sqrt{\frac{\log(4d/\delta)}{2k}}.$$

  If  $\lambda$ is chosen so that 
  $$
  \lambda \ge B(k,\delta) 
  + \bar  b(t_{n, \tilde k(\delta/2)}),
  $$
  then the bound \eqref{eq:predError_slow_largeLambda} on the prediction error holds true 
  with  probability at least $1-\delta$.
  In particular if $k/n$ is small enough so that
  $$\bar b(t_{n, \tilde k(\delta/2)}) \le B(k,\delta)$$ and if
  $\lambda \in [2B(k,\delta),~ 2C B(k,\delta)]$ for some $C>1$, then
    with probability at least $1-\delta$,  the prediction errors  $\mb W( \widehat{\boldsymbol{\beta}} - \boldsymbol{\beta}^*)$ satisfy
    \begin{equation}
      \label{eq:goodLambda_prederror_slow}
      \frac{1}{k} \|\mb W( \widehat{\boldsymbol{\beta}} - \boldsymbol{\beta^*})\|_2^2  \le 24\,  C \, M_\varepsilon \,\| \boldsymbol{\beta^*}\|_1\sqrt{\frac{\log(4d/\delta)}{2k}}. 
    \end{equation}
  

  \end{theorem}

\begin{remark}[Choice of $k$ and $\lambda$]
Theorem~\ref{theo:exlasso} suggests choosing $\lambda$ of order $O(\sqrt{\log(d)/k})$. While Lepski-type or adaptive validation methods are theoretically viable, cross-validation is often preferred in practice. The latter method proves successful in our experiments.  Further theoretical investigation is needed and left for future work.
\end{remark}


\subsection{Illustrative Numerical Experiments}
Our aim is to demonstrate the utility of introducing an $\ell_1$ penalty, as in XLASSO, in moderate-to-high dimensional settings for extrapolation on the covariates tail. We compare this approach to a baseline linear model trained on the angular component of extreme covariates. This baseline is a specific instance of the ROXANE algorithm proposed in \cite{huet2023regression}, namely an ERM algorithm without a penalty term, trained on the angles of extremes. For a comparison of this specific baseline with various other statistical and machine learning approaches, we refer to the experimental section of \cite{huet2023regression}. \vspace{0.2cm}

\noindent {\bf Simulated Data.}
Data are generated using an additive noise model that satisfies Assumptions~\ref{assum_1_huet},~\ref{assum_2_huet}, and~\ref{hyp:cont_reg_func} (Example 2.1 in \cite{huet2023regression}). Specifically, $\mb X \in \mathbb{R}^d$ follows a multivariate symmetric logistic distribution with dependence parameter $a=0.5$, making $\mb X$ simple max-stable and $\theta(\mb X)$ continuously distributed on $\sphere$, nearly uniform. The model is given by:

$$
Y = \langle \,\theta(\mb X), \,\boldsymbol{\beta}_0 \,\rangle ~+~ \frac{1}{\log(1+\|\mb X\|)} \langle \,\theta(\mb X), \,\boldsymbol{\beta}_1 \rangle + \varepsilon,
$$
where $\varepsilon$ is a bounded noise, specifically a truncated standard Gaussian noise on $[-2,2]$. We set $d=100$. The parameter $\beta_0$ has five entries equal to one, with the rest being zero. The 'bulk' parameter $\beta_1$ is a constant vector with all entries equal to one.

Training datasets of fixed size $n=10\,000$ are generated, with varying extreme sample sizes $k = \tau n$ for $\tau \in [0.011, 0.05]$. The parameter $\lambda$ is chosen for each $k$ and each replication using automatic cross-validation with the \texttt{LassoLarsCV} method in \texttt{scikit-learn}. The mean squared error is evaluated on a separate test dataset of size $1,000,000$, using the fraction $\tau_{\text{test}} = 0.01$ of the test data with the largest covariate norm. The procedure is repeated $N = 20$ times, and the average results along with the $(0.1-0.9)$ inter-quantile range are displayed in Figure~\ref{fig:simu_lasso}.
\begin{figure}[h]
  \centering
  \includegraphics[width=0.4\linewidth]{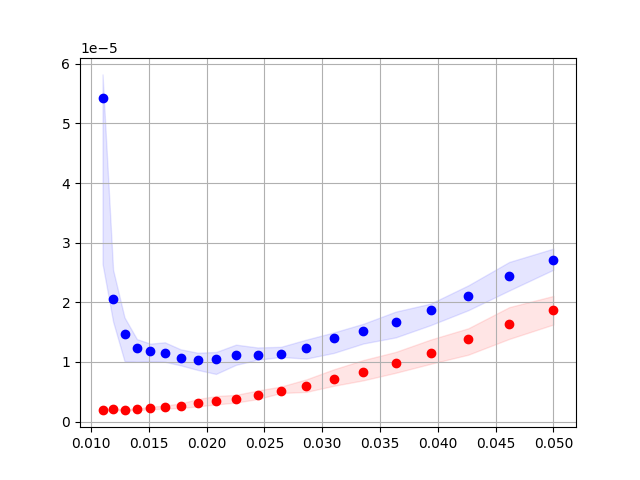}
  \caption[Synthetic data:mean squared error]{Simulated data in the additive noise model: mean squared error as a function of the ratio $\tau = k/n$. Red dots: XLASSO. Blue dots: linear model (witout penalization).}
  \label{fig:simu_lasso}
\end{figure}
\vspace{0.2cm}

\noindent {\bf Industry Portfolios Dataset.} This open access dataset has been used multiple times in the  \EVT literature \citep{meyer2024multivariate}, \citep{huet2023regression} as it provides an easy to manipulate example of relatively high-dimensional dataset ($49$ variables) with however a large number of observations $(n=13577)$.
In this work we take the \texttt{Trans} variable (transportation
sector) as a target, to be predicted given that the other variables
are large. In this example the covariate vector $\mb X$ has dimension
$d=48$. Based on the experiments in previous works mentioned above
bringing evidence of multivariate regular variation, we leverage Proposition~\ref{prop:exampleRescale} and we
consider the target $Y = Z/\|\mb X\|$ where $\|\point\|$ is the Euclidean
norm and $Z$ is the $\texttt{Trans}$ variable. The validity of the
boundedness assumption is investigated in the left panel of
Figure~\ref{fig:results_portfolio}, which reports the range (minimum
and maximum) of the values $\{Z_{(i)}/\|\mb X_{(i)}\| , i\le k \} $, as a
function of $k$. Stabilization of the empirical range of
$\mathcal{L}(Y ~|~\|\mb X\|\ge \|{\mb X}_{(k)} \|)$ brings strong evidence
that Assumption~\ref{assum_1_huet} is satisfied, especially above
large thresholds.  
\begin{figure}[h]
  \centering
  \includegraphics[width=0.3\linewidth]{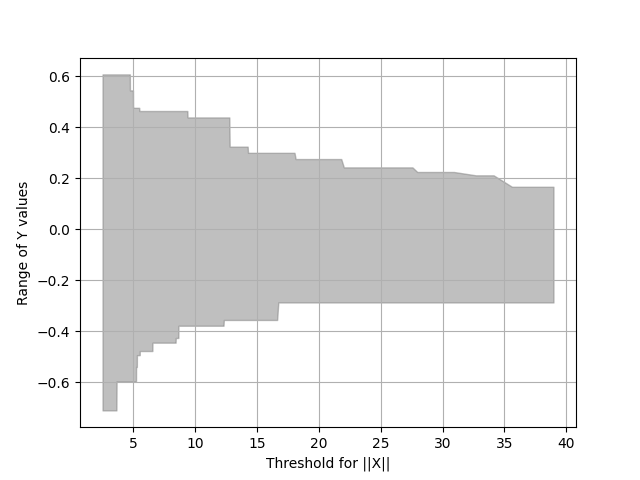}
  \includegraphics[width=0.3\linewidth]{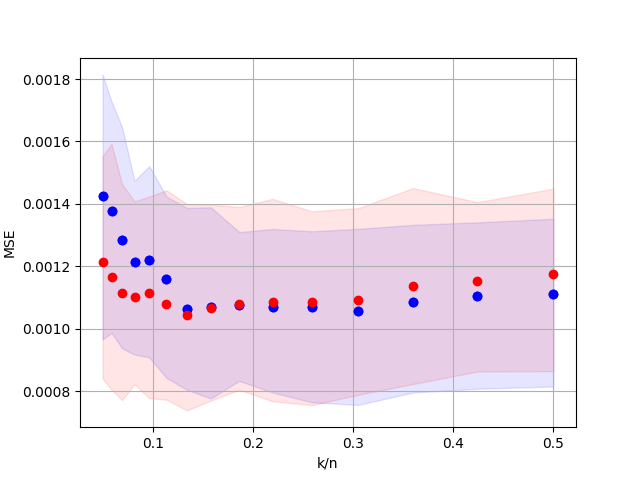}
  \caption{Left panel: Empirical support of $\mathcal{L}(Y ~|~\|\mb X\|\ge \|\mb X_{(k)} \|)$ versus the threshold $\|\mb X_{(k)} \|$.
    Right panel: Cross-validation error of XLASSO (red points) versus linear regression.\label{fig:results_portfolio}}
\end{figure}

The performance of XLASSO compared to linear regression (as in the experiments with synthetic data) is presented in the right panel of Figure~\ref{fig:results_portfolio}. Cross-validation is employed to evaluate the error. Over \( N = 50 \) independent experiments, a test set and a training set are randomly selected, with the test set being larger (\(0.8 \times n\), where \( n \) is the number of observations) to facilitate the evaluation of generalization error above thresholds potentially unseen in the training set. The empirical quantile level at the testing step is set at \( 1 - \tau_{\text{test}} \) with \( \tau_{\text{test}} = 0.005 \). At the training step, the number \( k \) of the largest observations retained for training is \( k = \lfloor \tau n_{\text{train}} \rfloor = \lfloor \tau \times 0.2 \times n \rfloor \) for \( \tau \in [0.05, 0.5] \). Consistent with the results from simulated data, XLASSO improves out-of-distribution generalization performance, with the effect being more pronounced when the number \( k \) of training data retained is small.

\section{Conclusion}\label{sec:perspectives}
In this article, we have endeavored to present an overview of some recent results combining extreme value theory and statistical learning. Our main objective was to demonstrate that it is possible to bring these two fields together in a common, nonparametric and nonasymptotic framework, and to highlight new methodological issues. The latter are inherent in the role played by the assumption of regular variation and the treatment of the resulting biases. Compared to traditional statistical learning techniques, additional standardization techniques are required to learn in tail regions. Perhaps most importantly, while the larger the training dataset, the lower the impact of statistical error on ERM methods in usual machine learning, the choice of the fraction $k$ of extreme examples is subject to a new trade-off: if it is too small, the frequentist principle of statistical learning cannot be effective, and if it is too large, the limiting behavior of distribution tails promised by multivariate regular variation is not well captured. We hope that this article will pave the way for further work on these methodological issues, which are likely to be found in many other predictive learning problems in tail regions, and lead to the development of successful algorithms.




\section*{Declarations}


\begin{itemize}
\item Funding:  Anne Sabourin's research was partially funded by the ANR PRC grant EXSTA, ANR-23-CE40-0009-01.  
\end{itemize}






\bibliographystyle{apalike}
\bibliography{bibli-general_1}


\appendix
\section{Proof of  Proposition~\ref{prop:exampleRescale} } 
\label{ap:sectionRegression}

The following  lemma is key to the proof of Proposition~\ref{prop:exampleRescale}.
The required conditions on the norms on $\rset^d$ and $\rset^{d+1}$ in the statement  hold true for any $\ell_q$ norm, $q\in(1, \infty]$. Below, we use the same notation  $\|\point\|$ for a norm on $\rset^d$ and on $\rset^{d+1}$ such that the canonical basis vectors have unit norm and $\|(\mb x,0)\| = \|\mb x\|$ for $\mb x\in\rset^d$.
\begin{lemma}[Conditioning on one component]\label{lem:RV_oneCOmponent_conditioning}
  Let $(\mb X,Z)$ be a random pair in $\rset^{d}\times \rset$, such that $\PP(\mb X=0)= 0$. 

  \begin{enumerate}
  \item Let  $(\mb X,Z)$ be  regularly varying, with limit distribution $\Pi_\infty$ supported on $\{(\mb x,z):\|(\mb x,z)\|>1\}$, defined by 
    \begin{equation}
      \label{eq:cv_xz_given_xz}
     \mathcal{L}\Big(t^{-1}(\mb X,Z) ~ |~ \|(\mb X,Z)\|>t \Big) \wto \Pi_\infty.  
    \end{equation}
    Assume additionally that $\Pi_\infty\{(\mb x,z): \|\mb x\|>1 \}>0$. 

    Then it also holds that 
 \begin{equation}
      \label{eq:cv_XZ_given_X}
     \mathcal{L}\Big(t^{-1}(\mb X,Z) ~ |~ \|\mb X\|>t \Big) \wto \tilde P_\infty 
    \end{equation}
    for some limit distribution $\tilde P_\infty$  supported on
$\{(\mb x,z):\|\mb x\|>1\}$.  In addition
$\tilde P_\infty$ and $\Pi_\infty$ are related  through the following identity,
  \begin{equation*}
    \tilde P_\infty(\point) = \Pi_\infty\big(\point)
    /\Pi_\infty \{ (\mb x,z): \|\mb x\|>1 \}. 
  \end{equation*}
  In other words if
  $(\tilde{ \mb X}_\infty, \tilde Z_\infty)\sim \tilde P_\infty$ and if
  $(\mb X'_\infty, Z'_\infty) \sim \Pi_\infty $, then
  $$
  \mathcal{L} (\tilde{\mb  X}_\infty, \tilde Z_\infty) =
  \mathcal{L} \Big( ( \mb X'_\infty, Z'_\infty)   ~|~  \|\mb X'_\infty\|>1\Big).
  $$
    
\item Conversely, let $(\mb X,Z)$ be a random pair and let $\tilde P_\infty$ be a distribution such that the   limit relation~\eqref{eq:cv_XZ_given_X} holds, and 
  additionally, assume  that the ratio $|Z|/\|\mb X\|$ is almost surely bounded by
  some constant $ M >0$. Then necessarily $\tilde P_\infty$ is
  supported on the truncated cone
  $$\tilde{ \mathcal{C}} = \{(\mb x,z): \|\mb x\|> 1 , |z|\le M\|\mb x\|\},$$ and
  then also \eqref{eq:cv_xz_given_xz} holds, with a limit distribution $\Pi_\infty$ is
  supported on the truncated cone
  $$\mathcal{C}' = \{( \mb x,z): \|(\mb x,z)\|> 1 , |z|\le M\|\mb x\|\}.$$
  

    
  
  \end{enumerate}
\end{lemma}
\begin{proof}
{\bf 1. }  
That~\eqref{eq:cv_xz_given_xz} implies \eqref{eq:cv_XZ_given_X}  is an easy exercise: 
For any measurable set $A\subset \{({\mb x},z): \|x\|>1\}$ 
such that $\Pi_\infty(\partial A) =0$, it holds that $\|{\mb X}\|>t \Rightarrow \|({\mb X},Z)\|>t$, thus 
  \begin{align*}
    \PP\Big(t^{-1}({\mb X},Z) \in A~|~\|{\mb X}\|>t\Big)&  = \PP\Big(t^{-1}({\mb X},Z) \in A~|~\|({\mb X},Z)\|>t\Big)\times
                                              \frac{ \PP \Big(\|({\mb X},Z)\|>t \Big)}{
                                              \PP\Big(\|({\mb X})\|>t\Big)}\\
    &\xrightarrow[t\to\infty]{}
    \Pi_\infty(A) \times \frac{ 1 }{\Pi_\infty\{ ({\mb x},z): \|\mb x\|>1 \} } . 
  \end{align*}
  This proves that~\eqref{eq:cv_xz_given_xz} implies \eqref{eq:cv_XZ_given_X}.
  
{\bf 2.}
Conversely, assume that~\eqref{eq:cv_XZ_given_X} holds and that $|Z|/\|{\mb X}\|<M$ almost surely. 
The fact that  the support of  $\tilde P_\infty$ is included in the truncated cone $\tilde{ \mathcal{C}}$ derives immediately from the assumptions.  
Also, with probability one, 
  $\|({\mb X},Z)\| \le\|{\mb X}\| +  |Z|  \le (M+1)\|{\mb X}\|$, 
  thus $\| {\mb X} \| \ge (M+1)^{-1} \| ({\mb X},Z)\|$. 
  A similar argument shows the following inclusions regarding  the truncated cones defined in the statement, 
  $$
(M+1) \mathcal{C}' \subset  \tilde{ \mathcal{C}} \subset { \mathcal{C}'}. 
  $$
 Thus, 
for any measurable set $A\subset \mathcal{C}'$ such that $\tilde P_\infty( (M+1)\partial A)= 0$,  
\begin{align*}
&  \PP\Big( t^{-1} ({\mb X},Z) \in A    ~|~  \| ({\mb X},Z) \| > t\Big) \\
  & = \PP \Big( t^{-1} ({\mb X},Z) \in A    ~|~  \| {\mb X} \|> t/(M+1) \Big) \times
    \frac{\PP( \|{\mb X}\| > t/(M+1) )}{ \PP( \|({\mb X},Z)\| > t ) } \\
    & \overset{u = t/(M+1)}{=} \PP \Big( u^{-1} ({\mb X},Z) \in (M+1)A    ~|~  \|{\mb X}\| > u \Big) \times
      \frac{\PP( \|{\mb X}\| > u )}{ \PP( \|({\mb X},Z)\| > u (M+1) ) } \\
& \xrightarrow[u\to \infty]{} \frac{\tilde P_\infty ((M+1)A) }{
  \tilde P_\infty ((M+1) \{ ({\mb x},z): \|({\mb x},z)\|>1 \})} 
\end{align*}
The proof is complete upon showing using  standard arguments  that continuity sets of $\tilde P_\infty$ and $\Pi_\infty$ are invariant under multiplication by a constant greater than one.
\end{proof}

\paragraph{\bf Proof of Proposition~\ref{prop:exampleRescale}}~
{\bf 1.}  Under the assumptions of Proposition~\ref{prop:exampleRescale}, \eqref{eq:cv_XZ_given_X} holds true by virtue of Lemma~\ref{lem:RV_oneCOmponent_conditioning}.1. 
The function $\varphi$ defined on $\{({\mb x},z):  x \neq 0\}$ by $\varphi({\mb x},z) = ({\mb x}, z/\|\mb x\|)$ is continuous. The continuous mapping theorem combined with~\eqref{eq:cv_XZ_given_X} then yields
$$
\mathcal{L}\Big(\varphi\big[ t^{-1}({\mb X},Z) \big] ~|~ \|{\mb X}\|>t \Big)\wto \tilde P_\infty\circ \varphi^{-1}, 
$$
where  $\tilde P_\infty$ is defined in Lemma~\ref{lem:RV_oneCOmponent_conditioning}.
However $\varphi\big[ t^{-1}({\mb X},Z) \big] = (t^{-1}{\mb X}, Z/\|{\mb X}\|) =(t^{-1}{\mb X}, Y)$ almost surely, and we obtain
$$
\mathcal{L}\Big(\big( t^{-1}{\mb X}, Y 
\big) ~|~ \|{\mb X}\|>t \Big)\wto \tilde P_\infty\circ \varphi^{-1}.
$$
 The latter display is precisely  Assumption~\ref{assum_2_huet} with $P_\infty = \tilde P_\infty\circ \varphi^{-1}$. Finally the relationship between $\Pi_\infty$ and $P_\infty$ derives immediately from the relationship between $\Pi_\infty$ and $\tilde P_\infty$ stated in Lemma~\ref{lem:RV_oneCOmponent_conditioning}. 
In other words, 
  $$
  P_\infty= \mathcal L \Big( 
  ({\mb X}'_\infty, ~ Z_\infty/\|{\mb X}'_\infty \|) ~|~ \|{\mb X}'_\infty\|>1 \Big),
  $$
  where    $({\mb X}'_\infty, Z_\infty)\sim \Pi_\infty$.

  {\bf 2.}  Conversely, let $(\mb X, Y)$ be a random pair such that
  $\mathcal{L}( t^{-1}({\mb X},Y)~|~ \|{\mb X}\|>t )\to P_\infty $
  (Assumption~\ref{assum_2_huet}) and $|Y|$ is bounded by $M>0$
  (Assumption~\ref{assum_1_huet}). Notice that the continous mapping
  $\varphi$ 
  is in fact a homeomorphism from
  $\rset^d\setminus\{0\} \times \rset$ onto itself. Hence,  the
  inverse mapping
  $\varphi^{-1}: ({\mb x},y)\mapsto (\mb x, \|\mb x\|y)$ is also
  continuous on $\rset\setminus\{0\} \times \rset$. Thus, the
  continuous mapping theorem implies that
 $$
 \mathcal{L}\Big(\varphi^{-1} (t^{-1}{\mb X},Y) ~|~ \|{\mb X}\|>t \Big)\wto
 P_\infty\circ\varphi.  
 $$
 Now, $\varphi^{-1} (t^{-1}{\mb X},Y) = t^{-1}({\mb X},Z)$ where $Z =\|{\mb X}\| Y$, so that the above display is equivalent to the convergence~\eqref{eq:cv_XZ_given_X} in Lemma~\ref{lem:RV_oneCOmponent_conditioning}. Using statement $2$ from the latter lemma, 
 we obtain that $({\mb X},Z)$ is regularly varying in $\rset^{d+1}$. 
\qed

{\bf 3}. Notice first that regular variation of the density $\pi$ for
the pair $({\mb X},Z)$ implies regular variation in the classical
sense \citep{de1987regular}, \citep{cai2011estimation}; thus, the assumptions of
the statements imply those of the first statement, in particular a
limit distribution $\Pi_\infty$ for the pair $({\mb X},Z)$ exists. In
view of statement $1$, the pair $(\mb X, Y)$ satisfies
Assumption~\ref{assum_2_huet}. 

The expression $p(\mb x,y) = \|{\mb x}\|\pi({\mb x},\|{\mb x}\| y )$ is a simple change of variable formula. Now, in order to prove~\eqref{eq:reg_var_density_xy}, uniform convergence on $\sphere\times \rset$ is sufficient (see \eg  the proof of Proposition 2.2 in \cite{huet2023regression}).  Letting
$p_\infty({\mb x},y) = \|{\mb x}\| \pi_\infty({\mb x}, \|{\mb x}\| y)$ as in the statement,  we have
\begin{multline*}
   \sup_{\boldsymbol{\omega} \in \sphere, y\in\rset}\big| b(t) t^d \|t\boldsymbol{\omega}\| p(t\boldsymbol{\omega}, y) - p_\infty(\boldsymbol{\omega}, y) \big\| =  \sup_{\boldsymbol{\omega} \in \sphere, y\in\rset}
    \big| b(t) t^{d+1} \pi(t\boldsymbol{\omega}, t y) - \pi_\infty(\boldsymbol{\omega},y) \big| \\
\le    \sup_{\|({\mb x},z)\|>t}
    \big| b(t) t^{d+1} \pi(t{\mb x}, t z) - \pi_\infty({\mb x},z) \big| \xrightarrow[t\to\infty]{} 0, 
 \end{multline*}
which proves uniform convergence~\eqref{eq:reg_var_density_xy}. With the  additional assumption that $\pi_\infty$ is lower bounded, the conditions of applications of Proposition 2.1-(iii) in \cite{huet2023regression} are satisfied, so that Assumption~\ref{hyp:cont_reg_func} also holds true. 

\section{Proof of Proposition~\ref{prop:deviations_residual}}\label{sec:proofPropDeviationsResidualLasso}

 Recall  the $i${th} entry of the residual vector is ${\mb e}_{i} = b({\mb X}_{(i)}) + \varepsilon_{(i)}, i\le k$, and notice $|b({\mb X}_{(i)})|\le \bar b(\|{\mb X}_{(k)} \|)$. In addition, the $(i,j)${th} entry of $\mb W$ satisfies   $|W_{i,j}|\le 1$. 
  We thus decompose the error as
  \begin{align}
    \|\mb W^\top \mb e\|_\infty/k
    &= \max_{j\le d}  \frac{1}{k} \Big|
      \sum_{i\le k} W_{i,j} \big( b({\mb X}_{(i)}) + \sigma({\mb X}_{(i)})\varepsilon_{(i)}\big) \Big|\nonumber \\
    & \le  \max_{j\le d}  \frac{1}{k} \Big|
      \sum_{i\le k} W_{i,j} \sigma({\mb X}_{(i)})\varepsilon_{(i)} \Big| +
       \bar b(\|{\mb X}_{(k)} \|). 
      \label{eq:decompos-residual} 
  \end{align}
  To control the second term in the above decomposition, we  rely on Lemma~\ref{lem:deviationEmpQuantiles} below.  By construction, the function $\bar b$ is nonincreasing. We obtain that with probability $1-\delta/2$, 
  \begin{equation}
    \label{eq:bound_bterm_residual}
    \bar b(\|{\mb X}_{(k)}\|)\le \bar b(t_{n, \tilde k(\delta/2)}) , 
  \end{equation}
  where $t_{n,\kappa}$ denotes the $1-\kappa/n$ quantile of $\|{\mb X}\|$ and
  $$\tilde k(\delta) =  k\Big(1 + \sqrt{\frac{3\log(1/\delta)}{k}} +  \frac{3\log(1/\delta)}{k} \Big).$$

  We turn to the first term in the right-hand side of~\eqref{eq:decompos-residual}.  Fix $j\le d$. Our argument proceeds conditionally to  ${\mb X}_{1:n}:=({\mb X}_1,\ldots, {\mb X}_n)$. 
  Because the noise variables  $\varepsilon_i$'s are independent of the ${\mb X}_i$'s,  the permutation $(\cdot)$ of the index set $1,\ldots, n$,  corresponding to the ranks of the $\|{\mb X}_i\|'s$,  is also independent of the $\varepsilon_i$'s. The exchangeability  of the $\varepsilon_i$'s, and the fact that the $W_{ij}'s$ are a function of ${\mb X}_{1:n}$,  implies 
  $$
  \mathcal{L} \Big( W_{i,j}\sigma({\mb X}_{(i)})\varepsilon_{(i)} , i\le k ~~|~~{\mb X}_{1:n} \Big) =
  \mathcal{L}\Big( W_{i,j}\sigma({\mb X}_{(i)})\varepsilon_{i} , i\le k ~~|~~{\mb X}_{1:n} \Big).
  $$
  Thus,  letting $T_{i,j} = W_{i,j}\sigma({\mb X}_{(i)})\varepsilon_{(i)}$,   the random variables $(T_{i,j}, i\le k)$ are independent, conditionally to ${\mb X}_{1:n}$, where we used that the $\varepsilon_i$'s are also independent conditionally to ${\mb X}_{1:n}$.  
  
 Also $|T_{i,j}|\le M_\varepsilon$ almost surely,  and by independence,
 $$
 \EE(T_{i,j}\given {\mb X}_{1:n} ) = 
 W_{i,j}\sigma({\mb X}_{(i)}) \EE(\varepsilon_{(i)} \given {\mb X}_{1:n}) = 0.
 $$
 A direct
 application of McDiarmid's inequality (conditionally to ${\mb X}_{1:n}$)
   yields that for $t>0$, for fixed $j\le d$, almost surely,
  $$
  \PP\Big(\big| k^{-1}\sum_{i\le k} T_{i,j}\big| \ge t ~\big|~ X_{1:n} \Big) \le 2\exp\Big( \frac{-2 k t^2}{M_\varepsilon^2} \Big).
  $$
  Integrating the above display with respect to the law of the
  $W_{i,j}$'s and a union bound over $j\in\{1,\ldots, d\}$ immediately
  yields the tail bound
  $$
\PP\Big(\max_{j\le d}\big|  \frac{1}{k}\sum_{i\le k} W_{i,j}\varepsilon_{(i)}  \big| \ge t \Big) \le 2d\exp\Big( \frac{-2 k t^2}{M_\varepsilon^2} \Big). 
$$
An inversion of the above bound, combined with the decomposition~\eqref{eq:decompos-residual} concludes the proof. 
\qed

The following  lemma is used to control the bias term in the error decomposition~\eqref{eq:decompos-residual}. 
  \begin{lemma}[Deviation of empirical quantiles]\label{lem:deviationEmpQuantiles}
    Let $R$ be  continuous, real-valued random variable with distribution function $F$. Let $F^{-1}$ denote the generalized (left-continuous) inverse of $F$. Let $R_i, i \le n$ be an \iid sample according to $F$, and let $R_{(1)}\ge \cdots \ge R_{(n)}$ denote the associated (decreasingly ordered) order statistics. Then with
    $$\tilde k(\delta):= k\Big(1 + \sqrt{\frac{3\log(1/\delta)}{k}} +  \frac{3\log(1/\delta)}{k} \Big),$$
    with probability at least $1-\delta$, it holds that 
    $$
R_{(k)}\ge F^{{-1}}\Big( 1 - \frac{\tilde k(\delta)}{n} \Big). 
    $$
  \end{lemma}
  \begin{proof}
    Let $U_i = F(R_i)$. By our continuity assumption, the $U_i$'s form an independent sample of standard uniform random variable. It is known that the order statistics of such a sample a uniform random sample are sub-gamma; namely, as shown in \citet[Lemma 3.1.1.]{reiss2012approximate} we have that for $k\le n$,
\begin{equation}
  \label{eq:deviationOrderUni-reiss}
  \PP\bigg\{\frac{\sqrt{n}}{\sigma }  \Bigl(1- U_{(k)} - \frac{k}{n+1}\Bigr) \ge t \bigg\}\le
  \exp\bigg[ -\,\frac{t^2}{3\Big\{1 + t/( \sigma\sqrt{n} ) \Big\}}\bigg] ,
\end{equation}
with $\sigma^2 = \{1-k/(n+1) \} \{{k}/(n+1) \} \le k/n$. The above display follows immediately from the cited reference and the fact that $1-U_{(k)} \eqd U_{(n+1-k)}$). Hence,
\begin{align*}
  \PP( 1 - U_{(k)} - k/n>t ) & \le \PP( 1 - U_{(k)} - k/(n+1)>t ) \\
  &\le \exp \Big\{ -\,\frac{nt^2/\sigma^2}{3(1 + t/ \sigma^2  )}\Big\}\,.
\end{align*}
Inverting the above inequality  yields that with probability greater than $1-\delta$,
\begin{align*}
  1 - U_{(k)}& \le \frac{k}{n} + \sqrt{\frac{3 \sigma^2 \log(1/\delta)}{n} }+
               \frac{3\log(1/\delta)}{n} \\
          & = \frac{k}{n}\Big\{1 + \sqrt{\frac{3 \log(1/\delta)}{k} }+
            \frac{3\log(1/\delta)}{k}  \Big\}.      
\end{align*}
Because $F(x)\ge y \iff x\ge F^{-1}(y)$ for any $x\in\rset$ and $y\in(0,1)$, the above display yields the statement of the lemma. 
\end{proof}




\end{document}